\documentclass[12pt]{amsart}
\usepackage[margin=1in]{geometry}
\usepackage{pgfplots}
\usetikzlibrary{calc}
\usepackage{amsmath}
\usepackage{amsfonts}
\usepackage{amssymb}
\usepackage{epsfig}
\usepackage{euscript}
\usepackage{color}
\usepackage{tikz}
\usepackage [all,arc,curve,color, arrow, matrix,frame]{xy}

 \newtheorem{theorem}{Theorem}[section]
 \newtheorem{corollary}[theorem]{Corollary}
 \newtheorem{lemma}[theorem]{Lemma}
 \newtheorem{proposition}[theorem]{Proposition}
 \theoremstyle{definition}
 \newtheorem{definition}[theorem]{Definition}
  
 \theoremstyle{remark}
 \newtheorem{remark}[theorem]{Remark}

 \numberwithin{equation}{section}
\numberwithin{figure}{section}
\newtheorem{example}[theorem]{Example}

\newcommand{\be}{\begin{equation}}
\newcommand{\ee}{\end{equation}}
\newcommand{\bea}{\begin{eqnarray}}
\newcommand{\eea}{\end{eqnarray}}

\newcommand{\C}{\mathbb{C}}
\newcommand{\R}{\mathbb{R}}

\newcommand{\N}{\mathbb{N}}

\newcommand{\Real}{{\mbox{Re}}}
\newcommand{\Imag}{{\mbox{Im}}}

\newcommand{\spann}{{\mbox{span}}}
\newcommand{\ran}{{\mbox{ran}}}

\def\idty{{\mathchoice {\mathrm{1\mskip-4mu l}} {\mathrm{1\mskip-4mu l}} %
{\mathrm{1\mskip-4.5mu l}} {\mathrm{1\mskip-5mu l}}}}

\numberwithin{equation}{section}

\begin{document}
\title[On proper dissipative extensions]{The Proper Dissipative Extensions of a Dual Pair}

\author[C. Fischbacher]{Christoph Fischbacher$^1$}
\address{$^1$ School of Mathematics, Statistics and Actuarial Science, University of Kent, Canterbury, Kent CT2 7NF, UK}
\email{cf299@kent.ac.uk}

\author[S. Naboko]{Sergey Naboko$^2$}
\address{$^2$ Department of Math.\ Physics, Institute of Physics, St.\ Petersburg State University 1 Ulianovskaia, St.\ Petergof, St.\ Petersburg  198504, Russia
}
\email{sergey.naboko@gmail.com}

\author[I. Wood]{Ian Wood$^3$}
\address{$^3$ School of Mathematics, Statistics and Actuarial Science, University of Kent, Canterbury, Kent CT2 7NF, UK}
\email{i.wood@kent.ac.uk}

\date{}

%
\begin{abstract}
Let $A$ and $(-\widetilde{A})$ be dissipative operators on a Hilbert space $\mathcal{H}$ and let $(A,\widetilde{A})$ form a dual pair, i.e. $A\subset\widetilde{A}^*$, resp.\ $\widetilde{A}\subset A^*$. We present a method of determining the proper dissipative extensions $\widehat{A}$ of this dual pair, i.e. $A\subset \widehat{A}\subset\widetilde{A}^*$ provided that $\mathcal{D}(A)\cap\mathcal{D}(\widetilde{A})$ is dense in $\mathcal{H}$.
 Applications to symmetric operators, symmetric operators perturbed by a relatively bounded dissipative operator and more singular differential operators are discussed.
 Finally, we investigate the stability of the numerical range of the different dissipative extensions.

\end{abstract}

\maketitle
\vspace{-7mm}
\section{Introduction} \label{sec:intro}  
The purpose of this paper is to develop a straightforward method for computing the proper dissipative extensions of a given dual pair of operators $(A,\widetilde{A})$, where $A$ and $(-\widetilde{A})$ are dissipative, under the mild assumption that $\mathcal{D}(A)\cap\mathcal{D}(\widetilde{A})$ is dense.\\
Numerous authors have contributed to the study of abstract extension problems for operators on Hilbert spaces, which goes at least back to von Neumann \cite{vNeumann} whose well-known von Neumann formulas provide a full characterization of all selfadjoint extensions of a given symmetric operator.
As it would be impossible to give an exhaustive overview here, let us just mention the results of Kre\u\i n, Vishik, Birman and Grubb (\cite{Krein}, \cite{Vishik}, \cite{Birman} and \cite{Grubb68}) who described all positive selfadjoint extensions of a given positive symmetric operator using positive selfadjoint operators on an auxiliary boundary space (cf.\ the survey \cite{Alonso-Simon} as well as the addendum acknowledging Grubb's contributions to the field \cite{Grubbfriendly}). Beyond that, Grubb's methods also allowed her to determine sectorial and $m$-sectorial extensions of positive symmetric operators \cite{Grubb68}.
For a much broader overview over the field, let us point the interested reader to the survey \cite{AT2009} and all the references therein (in particular also to the study of the extensions of linear relations rather than just operators).\\  
In his seminal paper \cite{Phillips59}, Phillips coined the term of a dissipative operator. He showed that dissipative operators always allow for maximally dissipative extensions, which are generators of $C_0$-semigroups of contractions. In order to determine these maximally dissipative extensions, he employed Kre\u\i n space methods as well as finding contractive extensions of the Cayley transform associated to the operator.\\
Lyantze and Storozh determined the maximally dissipative extensions of operators that one obtains by slightly varying abstract boundary conditions in the domain of certain symmetric operators such that the resulting operators are dissipative \cite{LyantzeStorozh}.\\
Moreover, for the sectorial case and for contributions towards extensions of dual pairs of operators, authors like Arlinski\u\i , Derkach, Kovalev, Malamud, Mogilevskii and {Tsekanovski\u\i} \cite{Arlinskii95, Kovalev, AT2005, DM91, MM97, MM99, MM02} have made many contributions using form methods and boundary triples in order to determine $m$-sectorial and $m$-accretive extensions (for an overview cf. \cite{Arli} and all the references therein). 
In particular for boundary triples, there has recently been a significant increase of interest with special attention towards their applications to PDE problems usually in the selfadjoint case \cite{BGW09, BHMNW09, BMNW08, GMZ07, GM08, GM09, GM09b, GM11, Gru08, Gru14, HMM13}.
Let us also point out examples, where explicit computations of maximally dissipative (resp.\ accretive) extensions for positive symmetric differential operators \cite{EvansKnowles85}, \cite{EvansKnowles86} and for sectorial Sturm-Liouville operators \cite{BrownEvans} have been performed.
Lastly, let us also mention that some recent developments in the theory of maximal monotone nonlinear operators can be found in \cite{Trostorff}. \\
We will proceed as follows:\\
In Section 2, we will give a few basic definitions and recall some useful results regarding dual pairs and dissipative operators and their extensions. \\
In Section 3, we introduce the common core property of a dual pair $(A,\widetilde{A})$, which ensures that the dual pair under consideration provides us with a convenient way of defining an operator $V$ that corresponds to the ``imaginary part" of $A$. \\
It will be the square root of the selfadjoint Kre\u\i n-von Neumann extension of $V$ -- denoted by $V_K^{1/2}$ -- which will play an important role for the results obtained in Section 4. The description of $V_K^{1/2}$ obtained by Ando and Nishio \cite{Ando-Nishio} will allow us to give a necessary and sufficient condition (Theorem \ref{prop:telemann}) for an extension of $(A,\widetilde{A})$ to be dissipative, which we only have to check on the space by which we extend the operator $A$ rather than on the whole domain of the extension. From this result, we proceed to give a description of all dissipative extensions of the dual pair $(A,\widetilde{A})$ in terms of contractions from one ``small" auxiliary space to another. We also generalize our results to the case that the common core property is not satisfied by the dual pair as long as $\mathcal{D}(A)\cap\mathcal{D}(\widetilde{A})$ is still dense. As a first application, we start by considering symmetric operators with relatively bounded dissipative perturbations and after that, we consider more singular dissipative operators -- our first examples being such that the associated imaginary part $V$ is already essentially selfadjoint and our last example being such that there is a family of selfadjoint extensions of $V$.\\
Finally, in Section 5, we find lower bounds for the numerical range of the dissipative extensions we have obtained and apply this result to the examples from the previous section.
\section*{Acknowledgements}
It is a pleasure to acknowledge Malcolm Brown for useful discussions and pointing out numerous references. C.F.\ would also like to thank Petr Siegl for providing him with the reference for the Hess-Kato Theorem \cite{Hess-Kato} and Eduard {Tsekanovski\u\i} for communicating various references.
C.F.\ acknowledges the UK Engineering and Physical Science Research Council (Doctoral Training Grant Ref.\ EP/K50306X/1) and the School of Mathematics, Statistics and Actuarial Science at the University of Kent for a PhD studentship.
S.N.\ was supported by the Russian Science Foundation (grant number 15-11-30007). S.N.\ also expresses his deep gratitude to the University of Kent at Canterbury for the hospitality and Marie Curie grant.

\section{Some definitions and conventions}
\subsection{Dissipative operators}
Let us begin with a few basic definitions and results on dissipative operators.
\begin{definition} 
 An operator $A$ on a Hilbert space $\mathcal{H}$ is said to be {\bf{dissipative}} if and only if it is densely defined and
 \begin{equation*}
 \Imag \langle f,Af\rangle\geq 0
 \end{equation*}
 for all $f\in\mathcal{D}(A)$. An operator $\widetilde{A}$ is called {\bf{antidissipative}} if and only if $(-\widetilde{A})$ is dissipative.
 \end{definition}

  Note that we have defined the scalar product $\langle\cdot,\cdot\rangle$ to be antilinear in the first and linear in the second component. Also note that we require $A$ to be densely defined for it to be dissipative. Finally, let us remark that any operator $A$, which is dissipative in the above sense, is also closable with its closure $\overline{A}$ being dissipative as well \cite{Phillips59}.

\begin{definition}
A dissipative operator $A$ is said to be {\bf{maximally dissipative}} if for any dissipative operator extension $A\subset A'$ we get that $A=A'$. 
\end{definition}
Let us remark at this point that the distinction between $m$-dissipative and maximally dissipative operators as it can be found in the literature (cf. e.g. \cite[Sec.\ 3]{Elst} for accretive operators) is not needed if one only considers densely defined dissipative operators as they coincide for this case. 
The following result is a well known fact:
\begin{proposition}[{\cite[Theorems 1.1.1, 1.1.2 and 1.1.3]{Phillips59}}] Let $A$ be dissipative. Then, the following are equivalent:
\begin{itemize}
\item $A$ is maximally dissipative.
\item There exists a $\lambda\in\C$ with $\Imag (\lambda)<0$ such that $\lambda\in\rho(A)$, where $\rho(A)$ denotes the resolvent set of $A$.
\item $\C^-:=\{z\in\C:\Imag (z)<0\}\subset\rho(A)$.
\item $(-A^*)$ is dissipative.
\item $iA$ is the generator of a strongly continuous semigroup of contractions on $\mathcal{H}$.
\end{itemize} \label{thm:bally}
\end{proposition}
Finally, let us state a lemma on by how many linearly independent vectors the domain of a given closed dissipative operator with finite defect index has to increase in order to obtain a maximally dissipative extension.  
\begin{lemma}[\cite{Crandall}] \label{prop:dampfnudel}
Let $A$ be a closed and dissipative linear operator on a separable Hilbert space $\mathcal{H}$ such that $\dim[\text{\emph{ran}}(A+i)]^\perp<\infty$. Moreover, let $A'$ be a dissipative extension of $A$. Then, $A'$ is maximally dissipative if and only if 
\begin{equation*}
\dim \mathcal{D}(A')/{\mathcal{D}(A)}=\dim [\text{\emph{ran}}(A+i)]^\perp\:.
\end{equation*}
\end{lemma} 
\subsection{Dual pairs}
Let us introduce the notion of a dual pair of operators (see also \cite{LyantzeStorozh} for more details). Given a densely defined closable operator $A$, it is a well known fact that another densely defined closable operator $\widetilde{A}$ can always be found such that $(A,\widetilde{A})$ forms a dual pair as can be seen from the trivial choice $\widetilde{A}:=A^*$.
\begin{definition} Let $(A,\widetilde{A})$ be a pair of densely defined and closable operators. We say that they form a {\bf{dual pair}} if 
\begin{equation*}
A\subset \widetilde{A}^*\quad\text{resp.}\quad \widetilde{A}\subset A^*\:.
\end{equation*}
In this case, $A$ is called a {\bf{formal adjoint}} of $\widetilde{A}$ and vice versa.
\end{definition}
Dual pairs can be thought of as a pair consisting of a ``maximal" operator (in our notation $\widetilde{A}^*$) and a ``minimal" operator (here: $A$). In this sense, any extension of $A$ that is a restriction of $\widetilde{A}^*$ can be interpreted as preserving the formal action of $\widetilde{A}^*$:
\begin{definition}
Let $(A,\widetilde{A})$ be a dual pair. An operator $A'$ is said to be a {\bf{proper}} extension of the dual pair $(A,\widetilde{A})$ if
\begin{equation*}
A\subset A'\subset\widetilde{A}^*\quad\text{resp.}\quad \widetilde{A}\subset (A')^*\subset A^*\:.
\end{equation*}
\end{definition}
Let us quote two useful results on the existence of proper extensions of certain dual pairs. The first proposition guarantees the existence of a proper extension of a dual pair $(A,\widetilde{A})$ with $\lambda\in\widehat{\rho}(A)$ and $\overline{\lambda}\in\widehat{\rho}(\widetilde{A})$, where $\widehat{\rho}(A)$ denotes the field of regularity of the operator $A$ (for a definition see e.g.\ \cite{Weid}). This applies in particular if $A$ is dissipative, which means that $\C^-\subset\widehat{\rho}(A)$ and if $\widetilde{A}$ is antidissipative, which implies $\C^+\subset\widehat{\rho}(\widetilde{A})$.
\begin{proposition}[{\cite[Chapter  II, Lemma 1.1]{Grubb68}}] \label{thm:EEvM}
Let $(A,\widetilde{A})$ be a dual pair with $\lambda\in\widehat{\rho}(A)$ and $\overline{\lambda}\in\widehat{\rho}(\widetilde{A})$. Then there exists a proper extension $\widehat{A}$ of $(A,\widetilde{A})$ such that $\lambda\in\rho(\widehat{A})$ and $\mathcal{D}(\widetilde{A}^*)$ can be expressed as
\begin{equation} \label{eq:malamud}
\mathcal{D}(\widetilde{A}^*)=\mathcal{D}(\overline{A})\dot{+}(\widehat{A}-\lambda)^{-1}\ker(A^*-\overline{\lambda})\dot{+}\ker(\widetilde{A}^*-{\lambda})\:.
\end{equation}
Likewise, we get the following description for $\mathcal{D}(A^*)$:
\begin{equation*}
\mathcal{D}({A}^*)=\mathcal{D}(\overline{\widetilde{A}})\dot{+}(\widehat{A}^*-\overline{\lambda})^{-1}\ker(\widetilde{A}^*-\lambda)\dot{+}\ker({A}^*-\overline{{\lambda}})\:.
\end{equation*}
\end{proposition}
The following proposition guarantuees the existence of a proper maximally dissipative extension for any dual pair $(A,\widetilde{A})$, where $A$ is dissipative and $\widetilde{A}$ is antidissipative. Up to a suitable multiplication by $i$, a proof for this can be found in \cite[Chapter IV, Proposition 4.2]{NagyFoias}.
\begin{proposition} \label{thm:friendly}
Let $(A,\widetilde{A})$ be a dual pair, where $A$ is dissipative and $\widetilde{A}$ is antidissipative. Then there exists a maximally dissipative proper extension of $(A,\widetilde{A})$.
\end{proposition}
Finally, let us introduce some convenient notation for complementary subspaces:
\begin{definition} Let $\mathcal{N}$ be a (not necessarily closed) linear space and $\mathcal{M}\subset\mathcal{N}$ be a (not necessarily closed) subspace. With the notation $\mathcal{N}//\mathcal{M}$ we mean any subspace of $\mathcal{N}$, which is complementary to $\mathcal{M}$, i.e.
\begin{equation*}
(\mathcal{N}//\mathcal{M})+\mathcal{M}=\mathcal{N}\quad\text{and}\quad(\mathcal{N}//\mathcal{M})\cap\mathcal{M}=\{0\}\:.
\end{equation*}
\end{definition}
\section{The common core property}

In many situations (including all of the examples that we are going to discuss in this paper) one considers dual pairs of operators, which are constructed by firstly defining them on a common core like, e.g. the compactly supported smooth functions, and then taking closures:
\begin{definition}
Let $(A,\widetilde{A})$ be a dual pair of closed operators. We say that it has the {\bf{common core property}} if $\overline{A\upharpoonright_{\mathcal{D}(A)\cap\mathcal{D}(\widetilde{A})}}=A$ and $\widetilde{A}=\overline{\widetilde{A}\upharpoonright_{\mathcal{D}(A)\cap\mathcal{D}(\widetilde{A})}}$.
\end{definition}
\begin{example}
Consider the dissipative momentum operator $T$ given by
\begin{align*}
T:\qquad\mathcal{D}(T)&=\{f\in H^1(0,1), f(0)=\rho f(1)\}, \quad
f\mapsto if'\:,
\end{align*}
where $|\rho|<1$. Here, $f'$ denotes the weak derivative of $f$. Its adjoint $T^*$ is given by
\begin{align*}
T^*:\qquad\mathcal{D}(T^*)=\{f\in H^1(0,1),\overline{\rho} f(0)= f(1)\},\quad
f\mapsto if'\:.
\end{align*}
Clearly, $(T,T^*)$ is a dual pair. However, since $\mathcal{D}:=\mathcal{D}(T)\cap\mathcal{D}(T^*)=\{f\in H^1(0,1), f(0)=f(1)=0\}$, this dual pair does not have the common core property, as $S:=\overline{T\upharpoonright_\mathcal{D}}$ is symmetric and a proper restriction of $T$.\\
More generally, let $S$ be a closed and symmetric (in particular densely defined) operator. Moreover, let $S'$ be any closed (not necessarily symmetric) extension of $S$ such that $S\subset S'\subset S^*$. This readily implies that $(S,S')$ is a dual pair. However, since $\mathcal{D}(S)\cap\mathcal{D}(S')=\mathcal{D}(S)$, we get $S=\overline{S'\upharpoonright_{\mathcal{D}(S)\cap\mathcal{D}(S')}}$. Thus, the only dual pair of this form, which has the common core property is $(S,S)$. Moreover, let $V\geq 0$ be $S^*$-bounded with $S^*$-bound less than $1$, which implies in particular that $V$ is $S'$-bounded with $S'$-bound less than $1$ (for a definition of relative boundedness, see e.g.\ \cite{Kato}). By the Hess-Kato Theorem \cite[Corollary 1]{Hess-Kato}, we have that $(S'+iV)^*=S'^*-iV\subset S^*-iV$. This implies again that any pair of the form $(S+iV,S'-iV)$ is a dual pair. However, again we have that the only dual pair which has the common core property is $(S+iV,S-iV)$.
\end{example}
The following lemma shows in particular that if we have a dual pair $(A,\widetilde{A})$ that has the common core property, where $A$ is dissipative, one can conclude that $\widetilde{A}$ is antidissipative.

\begin{lemma} Let $(A,\widetilde{A})$ be a dual pair of closed operators, which has the common core property. Moreover, let 
$\mathcal{N}_A:=\{\langle f,Af\rangle:f\in\mathcal{D}(A), \|f\|=1\}$ be the numerical range of $A$ and let $\mathcal{N}^*_{\widetilde{A}}:=\{{\overline{\langle f,\widetilde{A}f\rangle}}:f\in\mathcal{D}(\widetilde{A}), \|f\|=1\}$ be the complex conjugate of the numerical range of $\widetilde{A}$. Then, the closures of the numerical range of $A$ and the complex conjugate of the numerical range of $\widetilde{A}$ coincide:
\vspace{-2mm}
\begin{equation*}
\overline{\mathcal{N}_A}=\overline{\mathcal{N}^*_{\widetilde{A}}}\:.
\end{equation*}
\label{coro:keks}
\end{lemma}
\vspace{-8mm}
\begin{proof} Let $f\in\mathcal{D}(A)$ be normalized. Since $\mathcal{D}(A)\cap\mathcal{D}(\widetilde{A})$ is a core for $A$, there exists a normalized sequence $\{f_n\}_n\subset \mathcal{D}(A)\cap\mathcal{D}(\widetilde{A})$ such that $f_n\rightarrow f$ and $Af_n\rightarrow Af$ for $n\rightarrow\infty$. Using that $\langle f_n, A f_n\rangle=\overline{\langle f_n,\widetilde{A}f_n\rangle}$, we get that
\begin{equation*}
\lim_{n\rightarrow\infty}\overline{\langle f_n,\widetilde{A} f_n\rangle}=\lim_{n\rightarrow\infty}{\langle f_n,Af_n\rangle}={\langle f,Af\rangle}\:.
\end{equation*}
Since $\{\overline{\langle f_n,\widetilde{A}f_n\rangle}\}_n$ is a sequence of elements in $\mathcal{N}^*_{\widetilde{A}}$, we get that ${\langle f,Af\rangle}$ is a limit point of $\mathcal{N}^*_{\widetilde{A}}$, which means that
\begin{equation*}
{\mathcal{N}_A}\subset\overline{\mathcal{N}^*_{\widetilde{A}}}\:.
\end{equation*}
By similar reasoning, we get that
\begin{equation*}
\mathcal{N}^*_{\widetilde{A}}\subset\overline{\mathcal{N}_A},
\end{equation*}
which -- after taking closures -- yields the lemma.
\end{proof}

\begin{remark} If $A$ is closed and dissipative and $\mathcal{D}(A)\cap\mathcal{D}(A^*)$ is a core for $A$, i.e. $A=\overline{A\upharpoonright_{\mathcal{D}(A)\cap\mathcal{D}(A^*)}}$, we can define $\widetilde{A}:=\overline{A^*\upharpoonright_{\mathcal{D}(A)\cap\mathcal{D}(A^*)}}$, to construct a dual pair $(A, \widetilde{A})$, which has the common core property. This is in particular possible for the case that $\mathcal{D}(A)\subset\mathcal{D}(A^*)$ (cf. \cite[Corollary to Proposition IV, 4.2]{NagyFoias}).
\end{remark}
\section{The main theorem}
In this section, we will prove the main result, which can be written in a particularly nice form, if the common core property is satisfied. 
As any dissipative operator is closable with its closure being dissipative as well, it is necessary and sufficient to check dissipativity of an operator restricted to a core.
\begin{lemma} Let $A$ be a closed, densely defined operator and let $\mathcal{C}\subset\mathcal{H}$ be a core for $A$. Moreover, assume that $B$ is an extension of $A$, i.e. $A\subset B$ and $\mathcal{D}(B)=\mathcal{D}(A)\dot{+}\mathcal{M}$. Then, $\mathcal{C}\dot{+}\mathcal{M}$ is a core for $B$. \label{lemma:stollen}
\end{lemma}
\begin{proof}
Since $\mathcal{C}$ is a core for $A$, this means that for every $f\in\mathcal{D}(A)$ there exists a sequence $\{f_n\}_n\subset\mathcal{C}$ such that $f_n\rightarrow f$ and $Af_n\rightarrow Af$ and therefore for any element of $\mathcal{D}(B)\ni (f+ m)$, where $f \in\mathcal{D}(A)$ and $m\in\mathcal{M}$ we get
\begin{equation*} 
(f_n+ m) \rightarrow (f+ m) \quad\text{and}\quad B(f_n+m)=(Af_n+ Bm)\rightarrow (Af+Bm)=B(f+m)\:,
\end{equation*}
which is the desired result.
\end{proof}
For the following results, let us recall the definition of the Kre\u\i n-von Neumann extension of a symmetric non-negative operator:
\begin{definition}
Let $V$ be symmetric and non-negative operator, i.e. $\langle f,V f\rangle\geq 0$ for all $f\in\mathcal{D}(V)$. Then, the {\bf{Kre\u\i n-von Neumann extension}} of $V$, which we denote by $V_K$, is the smallest non-negative selfadjoint extension of $V$, i.e.
for any $\widehat{V}=\widehat{V}^*$ with $V\subset \widehat{V}$ and $\widehat{V}\geq 0$ we have that
\begin{equation*}
0\leq V_K\leq \widehat{V}\:.
\end{equation*} 
It is a well known fact that such an extension $V_K$ always exists and that it is unique (cf. \cite{Krein}).

 (Recall that for two non-negative selfadjoint operators $A$ and $B$ on a Hilbert space $\mathcal{H}$, the relation $A\leq B$ is defined as
\begin{equation*}
A\leq B :\Leftrightarrow \mathcal{D}(A^{1/2})\supset\mathcal{D}(B^{1/2})\:\:\text{and}\:\:\|A^{1/2}f\|\leq\|B^{1/2}f\|
\end{equation*}
for all $f\in\mathcal{D}(B^{1/2})$.)\\
For the special case that $V$ is strictly positive, i.e. there exists an $\varepsilon>0$ such $\langle f,Vf\rangle\geq\varepsilon\|f\|^2$ for all $f\in\mathcal{D}(V)$, we have the following characterization of $V_K$ \cite{Alonso-Simon}:
\begin{align*}
V_K:\quad\mathcal{D}(V_K)=\mathcal{D}(V)\dot{+}\ker V^*,\quad
V_K=V^*\upharpoonright_{\mathcal{D}(V_K)}
\end{align*}
and for $V_K^{1/2}$ we get
\begin{align} \label{eq:kreinformdomain}
V_K^{1/2}:\qquad\qquad\mathcal{D}(V_K^{1/2})=\mathcal{D}(V_F^{1/2})&\dot{+}\ker V^*\notag\\
\langle V_K^{1/2}(f+k),V_K^{1/2}(f+k)\rangle&=\langle V_F^{1/2}f,V_F^{1/2}f\rangle\:,
\end{align}
with $f\in\mathcal{D}(V_F^{1/2})$, where $V_F$ is the Friedrichs extension of $V$ and $k\in\ker V^*$.
\end{definition}
For the proof of the main theorem without having to assume that the imaginary part is strictly positive, we will make use of an equivalent description for non-negative $V_K^{1/2}$ proved by Ando and Nishio.
\begin{proposition}[T. Ando, K. Nishio, {\cite[Thm.\ 1]{Ando-Nishio}}]
Let $V$ be a non-negative closed symmetric operator.  The selfadjoint and non-negative square-root of the Kre\u\i n-von Neumann extension of $V$, which we denote by $V_K^{1/2}$ can be characterized as follows:
\begin{align*}
\mathcal{D}(V_K^{1/2})=\left\{h\in\mathcal{H}: \sup_{f\in\mathcal{D}(V):Vf\neq 0}\frac{|\langle h,Vf\rangle|^2}{\langle f,Vf\rangle}<\infty\right\}\:,\\
\text{for any}\:\: h\in\mathcal{D}(V_K^{1/2}):\quad\|V_K^{1/2}h\|^2= \sup_{f\in\mathcal{D}(V):Vf\neq 0}\frac{|\langle h,Vf\rangle|^2}{\langle f,Vf\rangle}\:.
\end{align*} \label{thm:nett}
\end{proposition}
\begin{remark} We draw the reader's attention to a slight difference in the way, Proposition \ref{thm:nett} was stated in \cite{Ando-Nishio}, where the supremum is taken over all $f\in\mathcal{D}(V)$ (without the extra condition that $Vf\neq 0$), which only makes sense if one assumes that $\ker V=\{0\}$. The extra condition $Vf\neq 0$ is a remedy for this problem and is a direct result from the reasoning of \cite{Ando-Nishio}.
\end{remark}
For our main theorem, we will make use of the fact that the dual pair under consideration has a common core $\mathcal{D}$, allowing us to define an ``imaginary part" on $\mathcal{D}$. It will therefore be helpful to show that the supremum in Proposition \ref{thm:nett} has to be taken only over $\mathcal{D}$.
\begin{lemma} Let $V$ be a non-negative closed symmetric operator and $\mathcal{C}$ be a core for $V$. Then, for any $h\in\mathcal{H}$ we have that 
\begin{equation*}
 \sup_{f\in\mathcal{D}(V):Vf\neq 0}\frac{|\langle h,Vf\rangle|^2}{\langle f,Vf\rangle}= \sup_{f\in\mathcal{C}: Vf\neq 0}\frac{|\langle h,Vf\rangle|^2}{\langle f,Vf\rangle}\:.
\end{equation*} \label{lemma:coresup}
\end{lemma}
\begin{proof} Let $s\in \R^+\cup\{\infty\}$ be defined as
\begin{equation*}
s:= \sup_{f\in\mathcal{D}(V):Vf\neq 0}\frac{|\langle h,Vf\rangle|^2}{\langle f,Vf\rangle}\:.
\end{equation*}
This means that there exists a sequence $\{f_n\}_n\subset\mathcal{D}(V)$ with $Vf_n\neq 0$ such that
\begin{equation*}
\lim_{n\rightarrow\infty}\frac{|\langle h,Vf_n\rangle|^2}{\langle f_n,Vf_n\rangle}=s\:.
\end{equation*}
On the other hand, since $\mathcal{C}$ is a core for $V$, for any $f_n\in\mathcal{D}(V)$, there exists a sequence $\{f_{n,m}\}_m\subset\mathcal{C}$ such that 
\begin{equation*}
\lim_{m\rightarrow\infty} f_{n,m}=f_n\quad\text{and}\quad \lim_{m\rightarrow\infty} Vf_{n,m}=Vf_n\:.
\end{equation*}
Thus, for any fixed $h\in\mathcal{H}$ and $f_n\in\mathcal{D}(V)$ such that $Vf_n\neq0$, we have also $\langle f_n,Vf_n\rangle\neq 0$ and therefore
\begin{equation*}
\lim_{m\rightarrow\infty}\frac{|\langle h, Vf_{n,m}\rangle|}{\langle f_{n,m},Vf_{n,m}\rangle}=\frac{|\langle h, Vf_{n}\rangle|}{\langle f_{n},Vf_{n}\rangle}\:.
\end{equation*}
 Hence, a diagonal sequence argument yields the lemma. 
\end{proof}
\begin{definition} Let $\mathcal{V}\subset\mathcal{D}(\widetilde{A}^*)//\mathcal{D}(A)$ be a subspace. Then, the operator $A_\mathcal{V}$ is defined as
\begin{align*}
A_\mathcal{V}:\qquad\mathcal{D}(A_\mathcal{V})=\mathcal{D}(A)\dot{+}\mathcal{V}, \quad
A_\mathcal{V}&=\widetilde{A}^*\upharpoonright_{\mathcal{D}(A_\mathcal{V})}\:.
\end{align*} \label{def:subspaceextension}
\end{definition}
\begin{theorem} Let $(A,\widetilde{A})$ be a dual pair of operators having the common core property, where $A$ is dissipative and let $\mathcal{D}\subset(\mathcal{D}(A)\cap\mathcal{D}(\widetilde{A}))$ be a common core for $A$ and for $\widetilde{A}$. Then, the operator $V:=\frac{A-\widetilde{A}}{2i}$ defined on $\mathcal{D}$ is a non-negative symmetric operator. Moreover, let $\mathcal{V}\subset \mathcal{D}(\widetilde{A}^*)//\mathcal{D}(A)$ be a linear space. Then, the operator $A_\mathcal{V}$ is dissipative if and only if $\mathcal{V}\subset\mathcal{D}(V_K^{1/2})$ and
\begin{equation*}
\Imag\langle v,\widetilde{A}^*v\rangle\geq \|V_K^{1/2}v\|^2\quad\text{for all}\quad v\in\mathcal{V}\:.
\end{equation*}
The operator $V_K$ does not depend on the specific choice of $\mathcal{D}$ as long as $\mathcal{D}\subset(\mathcal{D}(A)\cap\mathcal{D}(\widetilde{A}))$ is a common core for $A$ and $\widetilde{A}$.
\label{prop:telemann}
\end{theorem}
\begin{proof}
Since $\Imag\langle f,Af\rangle\geq 0$ for all $f\in\mathcal{D}(A)$, by Lemma \ref{coro:keks},  this implies that $\Imag\langle f,\widetilde{A}f\rangle\leq 0$ for all $f\in\mathcal{D}(\widetilde{A})$ and hence, $\widetilde{A}$ is anti-dissipative.
Next, let us show that $V$ is symmetric and non-negative. For any $f\in\mathcal{D}$ we get
\begin{equation} \label{eq:imaginary}
\langle f,V f\rangle=\frac{1}{2i}\left(\langle f,Af\rangle-\langle f,\widetilde{A} f\rangle\right)=\frac{1}{2i}\left(\langle f,Af\rangle-\langle Af,f\rangle\right)=\Imag\langle f,Af\rangle\geq 0
\end{equation}
by assumption. Let us now prove the criterion for dissipativity. By Lemma \ref{lemma:stollen}, it is sufficient to check dissipativity for all elements of $\mathcal{D}(A_\mathcal{V})$, which are of the form $f+ v$, where $f\in\mathcal{D}$ and $v\in\mathcal{V}$. Thus, it suffices to show that
\begin{equation*}
\Imag\langle f+v, \widetilde{A}^*(f+ v)\rangle\geq 0\quad\text{for all}\quad f\in\mathcal{D},\:v\in\mathcal{V}\:
\end{equation*}
if $\mathcal{V}\subset\mathcal{D}(V_K^{1/2})$ and $\Imag\langle v,\widetilde{A}^*v\rangle\geq\|V_K^{1/2}v\|^2$ for all $v\in\mathcal{V}$. Then by \eqref{eq:imaginary}:
\begin{align*} 
\Imag\langle f&+v, \widetilde{A}^*(f+ v)\rangle=\Imag\langle f,Af\rangle+\Imag\langle v,\widetilde{A}^* v\rangle+\Imag\left(\langle f,\widetilde{A}^*v\rangle+\langle  v, Af\rangle \right)\\
&{=}\langle f,V f\rangle + \Imag\langle v,\widetilde{A}^*v\rangle-\Imag\langle(A-\widetilde{A})f, v\rangle=\langle f,V f\rangle + \Imag\langle v,\widetilde{A}^*v\rangle-\Imag \langle 2iVf, v\rangle\:.
\end{align*}
Observe that for any given $v$, one can always consider $e^{i\vartheta}v$ instead of $v$, where $\vartheta\in [0,2\pi)$ is chosen such that $\Imag \langle 2iVf, e^{i\vartheta} v\rangle=-2\left|\langle Vf, v\rangle\right|$ without changing the other two terms, which means that showing 
\begin{equation} \langle f,Vf\rangle+\Imag\langle v,\widetilde{A}^*v\rangle-2|\langle Vf,v\rangle|\geq 0\quad \text{for all}\: f\in\mathcal{D},\: v\in\mathcal{V} \label{eq:positive}
\end{equation}
is necessary and sufficient for $A_\mathcal{V}$ being dissipative.\\
Let us begin by showing that $\mathcal{V}\subset\mathcal{D}(V_K^{1/2})$ and $\Imag\langle v,\widetilde{A}^*v\rangle\geq\|V_K^{1/2}v\|^2$ is sufficient for $A_\mathcal{V}$ to be dissipative. Thus, let us now assume that these two assumptions are satisfied. Since $V\subset\overline{V}\subset V_K$ and $\mathcal{D}(V)\subset\mathcal{D}(V_K)\subset\mathcal{D}(V_K^{1/2})$, this means that we can write $Vf=V_Kf=\left(V_K^{1/2}\right)\left(V_K^{1/2}f\right)$. We therefore get that 
\begin{align*}
\langle f,&Vf\rangle+\Imag\langle v,\widetilde{A}^*v\rangle-2|\langle Vf,v\rangle|=\|V_K^{1/2}f\|^2+\Imag\langle v,\widetilde{A}^*v\rangle-2|\langle V_K^{1/2}f,V_K^{1/2}v\rangle|\\&\geq \|V_K^{1/2}f\|^2+\Imag\langle v,\widetilde{A}^*v\rangle-2\|V_K^{1/2}f\|\|V_K^{1/2}v\|\geq \|V_K^{1/2}f\|^2+\|V_K^{1/2}v\|^2-2\|V_K^{1/2}f\|\|V_K^{1/2}v\|\\&=\left(\|V_K^{1/2}f\|-\|V_K^{1/2}v\|\right)^2\geq 0\:,
\end{align*}
Next, let us show that the condition $\mathcal{V}\subset\mathcal{D}(V_K^{1/2})$ is necessary for $A_\mathcal{V}$ to be dissipative. Thus, let us assume that $\mathcal{V}\not\subset\mathcal{D}(V_K^{1/2})$, i.e. that there exists a $v\in\mathcal{V}$ such that $v\notin\mathcal{D}(V_K^{1/2})$. Using that $\mathcal{D}(V)=\mathcal{D}$ is a core for $\overline{V}$, we have by Proposition \ref{thm:nett} and by Lemma \ref{lemma:coresup} that there exists a sequence $\{f_n\}_n\subset\mathcal{D}(V)$ with $Vf_n\neq 0$ and therefore $\langle f_n,Vf_n\rangle\neq 0$, such that
\begin{equation*}
\lim_{n\rightarrow\infty}\frac{|\langle v,Vf_n\rangle|}{\sqrt{\langle f_n,V f_n\rangle}}=+\infty\:.
\end{equation*}
Define the sequence $\{h_n\}_n\subset\mathcal{D}(V)$ by $h_n:={f_n}/{\sqrt{\langle f_n,V f_n\rangle}}$
and observe that 
\begin{equation*}
\frac{|\langle v,Vf_n\rangle|}{\sqrt{\langle f_n,V f_n\rangle}}=\frac{|\langle v,Vh_n\rangle|}{\sqrt{\langle h_n,V h_n\rangle}}\quad\text{and}\quad\sqrt{\langle h_n,Vh_n\rangle}=1 \quad\text{for all}\quad n\in\N\:.
\end{equation*}
From this we get that 
\begin{align*}
\lim_{n\rightarrow\infty}\left(\Imag \langle v,\widetilde{A}^*v\rangle+\langle h_n, V h_n\rangle-2|\langle Vh_n,v\rangle|\right)=\Imag\langle v,\widetilde{A}^* v\rangle+1-2\lim_{n\rightarrow\infty}\frac{|\langle v,V h_n\rangle|}{\sqrt{\langle h_n,V h_n\rangle}}=-\infty\:,
\end{align*}
which shows that Condition \eqref{eq:positive} can never be satisfied in this case.\\ Let us finish the proof by showing that $\Imag\langle v,\widetilde{A}^*v\rangle\geq \|V_K^{1/2}v\|^2$ for all $v\in\mathcal{V}$ is necessary for $A_\mathcal{V}$. By \eqref{eq:positive}, it suffices to show that for any $v\in\mathcal{D}(V_K^{1/2})$, there exists a sequence $\{g_n\}_n\subset\mathcal{D}(V)$ such that
\begin{equation} \label{eq:doesthejob}
2|\langle Vg_n,v\rangle|-\langle g_n,Vg_n\rangle\overset{n\rightarrow\infty}{\longrightarrow}\|V_K^{1/2}v\|^2\:.
\end{equation}
For the case $V_K^{1/2}v=0$, this sequence would just be given by $f_n=0$ for all $n$, therefore let us assume $V_K^{1/2}v\neq 0$ from now on. By Proposition \ref{thm:nett}, we know that there exists a sequence $\{f_n\}_n\subset\mathcal{D}(V)$ with $Vf_n\neq 0$ such that 
\begin{equation*}
\frac{|\langle v,Vf_n\rangle|^2}{\langle f_n,Vf_n\rangle}\overset{n\rightarrow\infty}{\longrightarrow}\|V_K^{1/2}v\|^2\:.
\end{equation*}
Define the positive numbers $\mu_n$ by $\mu_n:={|\langle v,Vf_n\rangle|}/{\langle f_n,Vf_n\rangle}$
and observe that the sequence $\{g_n\}_n$, where $g_n:=\mu_nf_n$, is exactly as required for \eqref{eq:doesthejob}:
\begin{align*}
2|\langle \mu_n Vf_n,v\rangle|&-\langle \mu_n f_n, \mu_n Vf_n\rangle=2|\langle Vf_n,v\rangle|\frac{|\langle Vf_n,v\rangle|}{\langle f_n,Vf_n\rangle}-\frac{|\langle Vf_n,v\rangle|^2}{\langle f_n,Vf_n\rangle^2}\langle f_n,Vf_n\rangle\\&=\frac{|\langle Vf_n,v\rangle|^2}{\langle f_n,Vf_n\rangle}\overset{n\rightarrow\infty}{\longrightarrow}\|V_K^{1/2}v\|^2\:.
\end{align*}
Finally, let us show that for $\mathcal{D}'\subset\mathcal{D}:= (\mathcal{D}(A)\cap\mathcal{D}(\widetilde{A}))$ both being common cores for $A$ and $\widetilde{A}$, we have that the Kre\u\i n-von Neumann extensions of 
$V_\mathcal{D'}=(A-\widetilde{A})/(2i)\upharpoonright_{\mathcal{D}'}$ and $V_\mathcal{D}=(A-\widetilde{A})/(2i)\upharpoonright_{\mathcal{D}}$ coincide. As we have already shown that $V_{\mathcal{D}'}$ and $V_\mathcal{D}$ are symmetric, it suffices to show that $\overline{V_{\mathcal{D}'}}=\overline{V_\mathcal{D}}$. Since $V_{\mathcal{D}'}\subset V_\mathcal{D}$, this will follow from $V_\mathcal{D}\subset\overline{V_{\mathcal{D}'}}$. Using that $\mathcal{D}'$ and $\mathcal{D}$ are cores for both $A$ and $\widetilde{A}$, we know that for any $f\in\mathcal{D}$, there exists a sequence $\{f_n\}_n\subset\mathcal{D}'$ such that $f_n\rightarrow f$ and $Af_n\rightarrow Af$. Moreover, since $f\in\mathcal{D}(\widetilde{A})$ and $\mathcal{D}'$ is a core for $\widetilde{A}$, this means that the sequence $\widetilde{A}f_n$ converges to $\widetilde{A}f$. Thus, for any $f\in\mathcal{D}$, there exists a sequence $\{f_n\}_n\subset\mathcal{D}'$ such that $V_{\mathcal{D}'}f_n=(2i)^{-1}(A-\widetilde{A})f_n\rightarrow (2i)^{-1}(A-\widetilde{A})f=V_\mathcal{D}f$, which implies that $V_\mathcal{D}\subset\overline{V_{\mathcal{D}'}}$.
\end{proof}
\begin{corollary} Let $(A,\widetilde{A})$ be a dual pair satisfying the assumptions of Theorem \ref{prop:telemann}. If for some $\lambda\in\C^-$ we have that 
\begin{equation} \label{eq:nodomain}
\ker(\widetilde{A}^*-\lambda)\cap\mathcal{D}(V_K^{1/2})=\{0\}\:,
\end{equation}
then there exists exactly one proper maximally dissipative extension of the dual pair $(A,\widetilde{A})$. \label{coro:uniquediss}
\end{corollary}
\begin{proof}
By Proposition \ref{thm:friendly}, we know that there exists a maximally dissipative extension $\widehat{A}$ and by Proposition \ref{thm:bally}, we know that $\C^-\in\rho(\widehat{A})$. Moreover, by \cite{Grubb68} we have that 
\begin{equation*}
\mathcal{D}(\widehat{A})=\mathcal{D}(A)\dot{+}(\widehat{A}-\lambda)^{-1}\ker(A^*-\overline{\lambda})
\end{equation*}
as well as
\begin{equation*}
\mathcal{D}(\widetilde{A}^*)=\mathcal{D}(A)\dot{+}(\widehat{A}-\lambda)^{-1}\ker(A^*-\overline{\lambda})\dot{+}\ker(\widetilde{A}^*-\lambda)\:.
\end{equation*}
By Theorem \ref{prop:telemann}, we know that $(\widehat{A}-\lambda)^{-1}\ker(A^*-\overline{\lambda})\subset\mathcal{D}(V_K^{1/2})$. As any other proper extension $A_\mathcal{V}$ of $(A,\widetilde{A})$ that is not a restriction of $\widehat{A}$ can be characterized by a subspace $\mathcal{V}$ that  without loss of generality we can assume to be contained in $(\widehat{A}-\lambda)^{-1}\ker(A^*-\overline{\lambda})\dot{+}\ker(\widetilde{A}^*-\lambda)$, where $\mathcal{V}\not\subset(\widehat{A}-\lambda)^{-1}\ker(A^*-\overline{\lambda})$, there needs to exist at least one element in $v\in\mathcal{V}$, which is of the form $v=(\widehat{A}-\lambda)^{-1}k_{\overline{\lambda}}+\widetilde{k}_\lambda$, where $k_{\overline{\lambda}}\in\ker(A^*-\overline{\lambda})$ and $\widetilde{k}_\lambda\in\ker(\widetilde{A}^*-\lambda)$ with $\widetilde{k}_\lambda\neq 0$. However, by \eqref{eq:nodomain}, we have that $v\notin\mathcal{D}(V_K^{1/2})$ which implies that $A_\mathcal{V}$ cannot be dissipative.
\end{proof}
\begin{remark} A corresponding result for sectorial operators was shown in \cite[Thm. 3.6.5]{Arli}.
\end{remark}
\begin{remark} In Example \ref{ex:eins} below, we will discuss an operator, for which Corollary \ref{coro:uniquediss} applies.
\end{remark}
\begin{remark} It is not necessary that \eqref{eq:nodomain} holds in order for a dual pair to have only one proper maximally dissipative extension as we will see in Example \ref{ex:zwei} below.
\end{remark}
\begin{theorem} In addition to the assumptions of Theorem \ref{prop:telemann}, assume that 
\begin{equation*}
\dim \mathcal{D}(\widetilde{A}^*)/\mathcal{D}(A)<\infty\:.
\end{equation*}
Moreover, let $\mathcal{W}:= (\mathcal{D}(\widetilde{A}^*)//\mathcal{D}(A))\cap\mathcal{D}(V_K^{1/2})$. Let the quadratic form $q$ be defined as
\begin{equation} \label{eq:form}
q(w):=\Imag\langle w,\widetilde{A}^* w\rangle-\|V_K^{1/2} w\|^2\:,
\end{equation}
which has domain $\mathcal{W}$ and let $M$ be the selfadjoint operator associated to the unique sesquilinear form induced by $q$ by polarization. Let us decompose  $\mathcal{W}=\mathcal{W}_+\oplus\mathcal{W}_0\oplus\mathcal{W}_-$, where $\mathcal{W}_+$ denotes the positive spectral subspace, $\mathcal{W}_0$ denotes $\ker M$ and $\mathcal{W}_-$ denotes the negative spectral subspace of $M$. Furthermore, define 
\begin{equation*}
M_\pm:=\pm MP_{\mathcal{W}_\pm}\:,
\end{equation*}
which allows us to write $M=M_+-M_-$. Note that $M_\pm > 0$ and that $M_+$ and $M_-$ are invertible on $\mathcal{W}_+$, resp.\ on $\mathcal{W}_-$. Let $C$ be a contraction ($\|C\|\leq 1$) from $\mathcal{W}_+\oplus\mathcal{W}_0$ into $\mathcal{W}_-$.
Then, there is a one-one correspondence between all pairs $(\mathfrak{M},C)$ , where $\mathfrak{M}$ is a subspace of $\mathcal{W}_+\oplus\mathcal{W}_0$ and $C$ is a contraction from $\mathcal{W}_+$ into $\mathcal{W}_-$ with $\mathcal{D}(C)=P_{\mathcal{W}_+}\mathfrak{M}$ and all proper dissipative extensions of $A$ via
\begin{align} \label{eq:extensioncontraction}
\mathcal{D}(A_{\mathfrak{M},C})&=\mathcal{D}(A)\dot{+}\{w+\sqrt{M_-}^{-1}C\sqrt{M_+}w, w\in\mathfrak{M}\}\notag\\
A_{\mathfrak{M},C}&=\widetilde{A}^*\upharpoonright_{\mathcal{D}(A_{\mathfrak{M},C})}\:.
\end{align} \label{prop:quadform}
Moreover, for an extension $\mathcal{D}(A_{\mathfrak{M},C})$ to be maximally dissipative, it is necessary that $\mathfrak{M}=\mathcal{W}_+\oplus \mathcal{W}_0$.
\end{theorem}
\begin{proof} By virtue of Theorem \ref{prop:telemann}, we firstly need to show that 
\begin{equation*}
q(v)\geq 0\quad\text{for all}\quad v\in \{w+\sqrt{M_-}^{-1}C\sqrt{M_+}w, w\in\mathfrak{M}\}
\end{equation*}
if $C$ is a contraction.
By definition of $M$ and $M_\pm$, we have that
\begin{align} \label{eq:contractionplusminus}
q(v)&=\langle v,Mv\rangle=\left\langle w+\sqrt{M_-}^{-1}C\sqrt{M_+}w,M\left( w+\sqrt{M_-}^{-1}C\sqrt{M_+}w\right)\right\rangle\notag\\
&=\langle w,M_+ w\rangle-\langle w,\sqrt{M_+}C^*\sqrt{M_-}^{-1}M_-\sqrt{M_-}^{-1}C\sqrt{M_+}w\rangle\notag\\
&=\langle w,\sqrt{M_+}(\idty-C^*C)\sqrt{M_+}w\rangle=\langle\sqrt{M_+}w,(\idty-C^*C)\sqrt{M_+}w\rangle\:,
\end{align}
which is non-negative if $C$ is a contraction on $ \sqrt{M_+}\mathfrak{M}=P_{\mathcal{W}_+}\mathfrak{M}=\mathcal{D}(C)$.\\ Let us now show that any proper dissipative extension has to be of this form. To this end, let $A'$ be a proper dissipative extension of $(A,\widetilde{A})$ and let $\mathfrak{M}'\subset\mathcal{W}$ be such that $\mathcal{D}(A')//\mathcal{D}(A)=\mathfrak{M}'$. Clearly, $\mathcal{W}_-\cap\mathfrak{M}'=\{0\}$, since otherwise we would have that
\begin{equation*}
q(w)=\langle w,Mw\rangle=-\langle w,M_-w\rangle<0
\end{equation*}
for some non-zero $w\in\mathcal{W}_-\cap\mathfrak{M}'$, which would violate the necessary condition as obtained from Theorem \ref{prop:telemann} for $A'$ to be dissipative. This means that any $w\in\mathfrak{M}'$ can be written as $w=w_-^\perp+w_-$ where $w_-^\perp\in\mathcal{W}_+\oplus\mathcal{W}_0$, $w_-^\perp\neq0$ and $w_-\in\mathcal{W}_-$ is possibly zero. Since $\mathcal{W}_-\cap\mathfrak{M}=\{0\}$, it is easy to see that $w_-$ is uniquely determined by $w_-^\perp$. Therefore, there exists a linear operator $B:P_{\mathfrak{M}'}(\mathcal{W}_+\oplus\mathcal{W}_0)\rightarrow\mathcal{W}_-$ such that $w=w_-^\perp+Bw_-^\perp$ for any $w\in\mathfrak{M}'$. Next observe that if for any such $w_-^\perp$ we have that $w_-^\perp\in\mathcal{W}_0$, it follows that $Bw_-^\perp=0$. If this were not true, we would get
\begin{equation*}
q(w_-^\perp+Bw_-^\perp)=\underbrace{\langle w_-^\perp,M_+w_-^\perp\rangle}_{=0}-\langle Bw_-^\perp,M_- Bw_-^\perp\rangle=-\langle Bw_-^\perp,M_- Bw_-^\perp\rangle\:,
\end{equation*}
which again would violate the necessary condition from Theorem \ref{prop:telemann} for $A'$ to be dissipative.
 Plugging this into the quadratic form $q$ yields:
\begin{align*}
q(w_-^\perp+Bw_-^\perp)&=\langle w_-^\perp,M_+w_-^\perp\rangle-\langle Bw_-^\perp,M_- Bw_-^\perp\rangle=\langle w_-^\perp,(M_+-B^*M_-B)w_-^\perp\rangle\\
&=\langle \sqrt{M_+}w_-^\perp,\left(\idty-\sqrt{M_+}^{-1}B^*\sqrt{M_-}\sqrt{M_-}B\sqrt{M_+}^{-1}\right)\sqrt{M_+}w_-^\perp\rangle\:,
\end{align*}
with the understanding that $\sqrt{M_+}^{-1}$ is defined only on $\ran\sqrt{M_+}=\ran M_+$. This is equivalent to saying that the operator $C:=\sqrt{M_-}B\sqrt{M_+}^{-1}$ is a contraction on $\sqrt{M_+}P_{\mathcal{W}_+}\mathfrak{M}'=P_{\mathcal{W}_+}\mathfrak{M}'$, or equivalently, $B=\sqrt{M_-}^{-1}C\sqrt{M_+}$, with $C$ being a contraction from $P_{\mathcal{W}_+}\mathfrak{M}'$ to $\mathcal{W}_-$. The condition that $\mathfrak{M}=\mathcal{W}_+\oplus\mathcal{W}_0$ for $A_{\mathfrak{M},C}$ to be maximally dissipative follows from the fact that one could always extend the operator $A_{\mathfrak{M},C}$ to $A_{\mathcal{W}_+\oplus\mathcal{W}_0, \widehat{C}}$, where $\widehat{C}$ is an extension of $C$ which is just set equal to zero on $(\mathcal{W}_+\oplus\mathcal{W}_0)\ominus\mathfrak{M}$.
\end{proof}
\begin{remark} For the case that the dual pair $(A,\widetilde{A})$ has only one unique maximally dissipative proper extension $\widehat{A}$, this means that $\widehat{A}=A_{\mathcal{W}_+\oplus\mathcal{W}_0, 0}$. In particular, for the case that the assumptions of Corollary \ref{coro:uniquediss} are satisfied, we get that $\mathcal{W}_-=\{0\}$ since $(\mathcal{D}(\widetilde{A}^*)//\mathcal{D}(A))\cap\mathcal{D}(V_K^{1/2})=\mathcal{W}_+\oplus\mathcal{W}_0$.
\end{remark} 
\begin{remark} Let us show that for a very special situation, the spaces $\mathcal{W_\pm}$ coincide with the defect spaces of a symmetric operator $S$. (As an example, take the momentum operator $i\frac{d}{dx}$ with domain $\{f\in H^1(\R), f(0)=0\}$, whose defect spaces are one-dimensional and spanned by exponential functions supported on different half-lines.) Assume that $S$ has finite-dimensional defect spaces $\mathcal{N}_\pm:=\ker(S^*\mp i)$. It is a well-known fact \cite{Weid} that
\begin{equation*}
\mathcal{D}(S^*)=\mathcal{D}(S)\dot{+}\mathcal{N}_+\dot{+}\mathcal{N}_-\:,
\end{equation*}
where $\mathcal{N}_\pm:=\ker(S^*\mp i)$ are the defect spaces. Assume in addition the rather restrictive condition that $\mathcal{N}_+\perp\mathcal{N}_-$ (orthogonal with respect to the Hilbert space inner product).
Choosing the dual pair $(S,S)$, which trivially has the common core property, we find that $V_K=0_\mathcal{H}$, with $V_K$ being defined as in Theorem \ref{prop:telemann}. Define
\begin{equation*}
q(v):=\Imag\langle v, S^*v\rangle\quad\text{with}\quad v\in\mathcal{N}_+\oplus\mathcal{N}_-\:.
\end{equation*}
A calculation shows that the operator $M$ associated to $q$ is given by $M=P_{\mathcal{N}_+}-P_{\mathcal{N}_-}$, i.e.\ $M_\pm=P_{\mathcal{N}_\pm}$, $\mathcal{W}_\pm=\mathcal{N}_\pm$ and $\mathcal{W}_0=\{0\}$. Thus, by Theorem \ref{prop:quadform}, all maximally dissipative extensions of such an operator $S$ are given by
\begin{align*}
\mathcal{D}(S_C)=\mathcal{D}(S)\dot{+}\{n_++Cn_+, n_+\in\mathcal{N}_+\},\quad
S_C=S^*\upharpoonright_{\mathcal{D}(S_C)}\:,
\end{align*}
where $C$ is any contraction into $\mathcal{N}_-$ such that $\mathcal{D}(C)=\mathcal{N}_+$. Thus, for the very special case $\mathcal{N}_+\perp\mathcal{N}_-$, this readily implies the von Neumann theory of selfadjoint/maximally dissipative extensions of symmetric operators. (cf. e.g. \cite[Thm.\ 8.12]{Weid},  for the selfadjoint and \cite[Theorem 2.4]{Kovalev},  for the more general maximally dissipative case)
\end{remark}
\begin{remark} For concrete problems, it seems to be not very practical to construct $\mathcal{W}_+,\mathcal{W}_0$ and $\mathcal{W}_-$ as well as $M_+$ and $M_-$. However, this result allows us to calculate the number of independent complex parameters one can expect to describe all proper maximally dissipative extensions of a dual pair, which is given by the number of parameters that describe all contractions $C$ from $\mathcal{W}_+$ into $\mathcal{W}_-$, which is equal to $\dim\mathcal{W}_+\cdot\dim\mathcal{W}_-\:.$\\
See also the operators considered in Example \ref{ex:firstorder} for a discussion of the spaces $\mathcal{W}_+, \mathcal{W}_-$ and $\mathcal{W}_0$ for a few concrete examples.
\end{remark}
\begin{remark} As a reference to \cite{MalamudOpHoles}, let us point out that this result means that we can characterize all proper dissipative extensions of such a dual pair using the terminology of \emph{operator balls}. For any three operators $Z,R_l,R_r\in\mathcal{B}(\mathcal{E})$, where $\mathcal{E}$ is an arbitrary Hilbert space, recall that the set of all operators $K\in\mathcal{B}(\mathcal{E})$ such that there exists a contraction $C$ from $\ran R_r$ to $\mathcal{D}(R_l)$ such that
\begin{equation*}
K=Z+R_lCR_r
\end{equation*}
is called an operator ball $\mathfrak{B}(Z,R_l,R_r)$ with center point $Z$, left radius $R_l$ and right radius $R_r$. With the identification $\mathcal{E}=\mathcal{W}$, $Z=P_{\mathcal{W}_+}+P_{\mathcal{W}_0}$, $R_l=\sqrt{M_-}^{-1}$ and $R_r=\sqrt{M_+}$ defined on $\mathcal{W}_-$, respectively on $\mathcal{W}_+$ and the result from Theorem \ref{prop:quadform}, we can characterize all proper dissipative extensions of a dual pair $(A,\widetilde{A})$ satisfying the assumptions of Theorem \ref{prop:quadform} via:
\begin{align}
A_K:\quad\mathcal{D}(A_K)=\mathcal{D}(A)\dot{+}\{Kw:w\in\mathcal{W}\},\quad
A_K=\widetilde{A}^*\upharpoonright_{\mathcal{D}(A_K)}\:,
\label{eq:opballs}
\end{align}
where $K\in\mathfrak{B}(P_{\mathcal{W}_+}+P_{\mathcal{W}_0},\sqrt{M_-}^{-1},\sqrt{M_+})$. \label{rem:opballs}
\end{remark}
\subsection{The non-common core case}
Let us now extend this idea to the case where the dual pair $(A,\widetilde{A})$ does not have the common core property. If we assume $\mathcal{D}(A)\cap\mathcal{D}(\widetilde{A})$ to still be dense, we can restrict $A$ and $\widetilde{A}$ to $\mathcal{D}(A)\cap\mathcal{D}(\widetilde{A})$ to obtain a dual pair of operators which has the common core property:
\begin{corollary} Let $A$ and $\widetilde{A}$ be a dual pair of operators, where $A$ is dissipative. Moreover, let $\mathcal{D}(A)\cap\mathcal{D}(\widetilde{A})$ be dense in $\mathcal{H}$. Define the operators $A'$ and $\widetilde{A}'$ as follows:
\begin{align*}
A':=\overline{A\upharpoonright_{\mathcal{D}(A)\cap\mathcal{D}(\widetilde{A})}}\quad\text{and}\quad
\widetilde{A}':=\overline{\widetilde{A}\upharpoonright_{\mathcal{D}(A)\cap\mathcal{D}(\widetilde{A})}}\:.
\end{align*}
Furthermore, let $V'_0$ denote the operator $\frac{1}{2i}(A'-\widetilde{A}')$ on $\mathcal{D}(A)\cap\mathcal{D}(\widetilde{A})$ and $V_K'$ its corresponding Kre\u\i n extension.\\
Now, let $\mathcal{V}\subset \mathcal{D}(\widetilde{A}'^*)//\mathcal{D}(A')$ be a subspace. The operator $A'_\mathcal{V}$ is a proper dissipative extension of the dual pair $A$ and $\widetilde{A}$ if and only if all of the following conditions are satisfied
\begin{itemize}
\item $\mathcal{V}\subset \mathcal{D}({V_K'}^{1/2})$
\item $\Imag\langle v,\widetilde{A}^*v\rangle\geq\|{V_K'}^{1/2}v\|^2$ for all $v\in\mathcal{V}$
\item $\mathcal{D}(A)\subset \mathcal{D}(A'_\mathcal{V})$
\item $\mathcal{V}\subset \mathcal{D}(\widetilde{A}^*)$\:.
\end{itemize}
\end{corollary}
\begin{proof} Since $\mathcal{D}(A)\cap\mathcal{D}(\widetilde{A})$ is dense, the operator $A\upharpoonright_{\mathcal{D}(A)\cap\mathcal{D}(\widetilde{A})}$ is a densely defined dissipative operator and thus closable. Moreover, since 
\begin{equation*}
\Imag\langle \psi,A\psi\rangle=\Imag \langle \widetilde{A} \psi,\psi\rangle=-\Imag\langle \psi,\widetilde{A}\psi\rangle\geq 0\quad\text{for all}\quad \psi\in\mathcal{D}(A)\cap\mathcal{D}(\widetilde{A})\:,
\end{equation*}
this shows that $\widetilde{A}\upharpoonright_{\mathcal{D}(A)\cap\mathcal{D}(\widetilde{A})}$ is a densely defined anti-dissipative operator. Thus, by construction, the operators $A'$ and $\widetilde{A}'$ are closed operators, which have the common core property. Moreover,
\begin{equation*}
A'\subset A\subset\widetilde{A}^*\subset {\widetilde{A}}'^*\:,
\end{equation*}
from which follows that any proper dissipative extension of the dual pair $(A,\widetilde{A})$ is a proper extension of the dual pair $(A',\widetilde{A}')$ as well. The corollary now follows from the observation that its first two conditions just correspond to an application of Theorem \ref{prop:telemann} for the dual pair $(A',\widetilde{A}')$ (which has the common core property) to ensure that $A'_\mathcal{V}$ is a dissipative extension of $A'$. The latter two conditions ensure that $A'_\mathcal{V}$ is not just a proper extension of the dual pair $(A',\widetilde{A}')$ but also of $(A,\widetilde{A})$.
\end{proof}
\begin{remark}
Since the dual pair $(A',\widetilde{A}')$ has the common core property and $A$ is a proper dissipative extension of this dual pair, Theorem \ref{prop:quadform} implies that there exists a contraction $C$ from $\mathcal{W}'_+$ into $\mathcal{W}'_-$ and a subspace $\mathfrak{M}'\subset \mathcal{W}'_+\oplus \mathcal{W}'_0$ such that $A=A'_{\mathfrak{M}',C}$, where the notation is the same as employed in \eqref{eq:extensioncontraction}. As any proper dissipative extension of the dual pair $(A,\widetilde{A})$ has to be a proper dissipative extension of $(A',\widetilde{A}')$ as well, for which Theorem \ref{prop:quadform} applies, this means that the problem of finding the proper dissipative extensions of $(A,\widetilde{A})$ is equivalent to determining $(\mathfrak{N},\widehat{C})$, where $\mathfrak{M}'\subset\mathfrak{N}$ and $\widehat{C}$ is a contractive extension of $C$  with the additional constraint that $A_{\mathfrak{N},\widehat{C}}\subset\widetilde{A}^*$ . For a full discussion of determining the contractive extensions of a given contraction, see \cite{AG}.
\end{remark}
\subsection{Illustrating examples}
In the following, we are going to apply our results to various ODE examples, which we have chosen to illustrate our results without having to worry too much about technicalities. 
\subsubsection{\bf{Weakly perturbed symmetric operators}}
As a first application of Theorem \ref{prop:telemann}, let us consider dual pairs of operators of the form $A=S+iV$ and $\widetilde{A}=S-iV$, where $S$ is closed and symmetric and $V$ is a positive symmetric operator, which has $S^*$-bound less than one.\footnote{Actually, we could consider dual pairs of the form $(S+D,S+\widetilde{D})$, where $(D,\widetilde{D})$ is a dual pair of dissipative/antidissipative perturbations, which are both relatively bounded with respect to $S^*$ with relative bound less than $1$.}
\begin{theorem} \label{prop:weaklyperturbed}
Let $S$ be a closed symmetric operator and $V$ be a non-negative symmetric operator with $S^*$-bound less than $1$. Moreover, let $\mathfrak{d}(A,\widetilde{A})$ denote the set of proper dissipative extensions of the dual pair $(A,\widetilde{A})$. Then, the set of all proper dissipative extensions of the dual pair $S+iV$ and $S-iV$ is given by
\begin{equation*}
\mathfrak{d}(S+iV,S-iV)=\{\widehat{S}+iV; \widehat{S}\in\mathfrak{d}(S,S)\}\:. 
\end{equation*}
\end{theorem}
\begin{proof} 
Firstly, let us apply Theorem \ref{prop:telemann} to the dual pair $(S,S)$, where $S$ is closed and symmetric. In this case, the operator $(S-S)/(2i)$ is identical to the zero operator on $\mathcal{D}(S)$, which has a unique bounded extension to the zero operator on the whole Hilbert space $\mathcal{H}$, i.e. $0_\mathcal{H}$. Thus, for any extension $S_\mathcal{V}$, where $\mathcal{V}\subset \mathcal{D}(S^*)//\mathcal{D}(S)$, we trivially have $\mathcal{V}\subset\mathcal{D}(0_\mathcal{H})=\mathcal{H}$. Thus, $\mathcal{V}$ needs only to satisfy the condition
\begin{equation} \label{eq:symmetrisch}
\Imag\langle v,S^*v\rangle\geq 0\quad\text{for all}\quad v\in\mathcal{V}\:.
\end{equation}
Next, let us consider the dual pair $(S+iV,S-iV)$.
By the Hess-Kato Theorem \cite[Corollary 1]{Hess-Kato}, we get that $(S-iV)^*=S^*+iV$, which we use together with Theorem \ref{prop:telemann}. By relative boundedness, we therefore have $\mathcal{D}((S-iV)^*)=\mathcal{D}(S^*)$ as well as $\mathcal{D}(S+iV)=\mathcal{D}(S)$, which means that we can choose $\mathcal{D}((S-iV)^*)//\mathcal{D}(S+iV)=\mathcal{D}(S^*)//\mathcal{}(S)$. Now, observe that
\begin{equation*}
\Imag\langle v,(S-iV)^*v\rangle=\Imag \langle v,(S^*+iV)v\rangle=\Imag \langle v,S^*v\rangle+\langle v,Vv\rangle
\end{equation*}
and that 
\begin{equation*}
\langle v,Vv\rangle=\|V_K^{1/2}v\|^2\quad\text{for all}\quad v\in\mathcal{D}(S^*)=\mathcal{D}(S^*+iV)\:,
\end{equation*}
which follows from relative boundedness of $V$ with respect to $S^*$. Hence, again we have that $\mathcal{V}\subset\mathcal{D}(V_K^{1/2})$ is always satisfied for any $\mathcal{V}\subset\mathcal{D}((S-iV)^*)//\mathcal{D}(S+iV)$. This implies that $\mathcal{V}$ only needs to satisfy
\begin{equation*}
\Imag \langle v, (S-iV)^*v\rangle\geq \|V_K^{1/2}v\|^2\quad\text{which is equivalent to}\quad\Imag \langle v,S^*v\rangle\geq 0\quad\text{for all}\quad v\in\mathcal{V}\:.
\end{equation*}
However, since this is equivalent to Condition \eqref{eq:symmetrisch}, we get that $(S+iV)_\mathcal{V}$ is dissipative if and only if $S_\mathcal{V}$ is dissipative.
\end{proof}
Let us start with the elementary example of a first order differential operator.
\begin{example} \label{ex:donnerkuchen}
Consider the closed symmetric operator on $L^2(0,1)$, which is given by
\begin{align*}
S:\quad\mathcal{D}(S)=\{ f\in H^1(0,1); f(0)=f(1)=0\},\quad f\mapsto if'\:,
\end{align*}
where $f'$ denotes the weak derivative of $f$. Its adjoint $S^*$ is given by
\begin{align*}
S^*:\quad\mathcal{D}(S^*)=H^1(0,1),\quad f\mapsto if'\:.
\end{align*}
Since for any $f\in\mathcal{D}(S^*)$, we have that
\begin{equation*}
\Imag\langle f,S^*f\rangle=\frac{1}{2}\left[|f(1)|^2-|f(0)|^2\right]\:,
\end{equation*}
it follows that all dissipative extensions of $S$ are given by
\begin{align*}
S_c:\quad\mathcal{D}(S_c):=\left\{f\in H^1(0,1);  f(0)=cf(1)\right\},\quad S_c=S^*\upharpoonright_{\mathcal{D}(S_c)}\:,
\end{align*}
where $c$ is any complex number such that $|c|\leq 1$. Using Lemma \ref{prop:dampfnudel}, it is in fact not hard to see that these extensions are also maximal.\\
Moreover, let $V$ be the selfadjoint maximal multiplication operator by a non-negative and non-zero $L^2$-function $V(x)$: 
\begin{align*}
V:\quad\mathcal{D}(V)=\left\{f\in L^2(0,1); \int_0^1 {|V(x)f(x)|^2}\text{d}x<\infty\right\},\quad \left(Vf\right)(x)={V(x)f(x)}\:.
\end{align*}
For example, one could pick $V(x)=x^{-\alpha}$ with $\boxed{0<\alpha< 1/2} $.
Using that $H^1(0,1)$ compactly embeds into the bounded continuous functions $C([0,1])$ we may use that by Ehrling's Lemma there exists for any $\varepsilon> 0$ a $C(\varepsilon)$ such that 
\begin{equation} \label{eq:ehrling}
\|f\|_\infty \leq\varepsilon\|f'\|+{C}(\varepsilon)\|f\|\:,
\end{equation}
for all $f\in H^1(0,1)$.
This allows us to show that $V$ is $S^*$-bounded with $S^*$-bound equal to zero:
\begin{equation*}
\|Vf\|_2\leq \|V\|_2\|f\|_\infty\overset{\eqref{eq:ehrling}}{\leq} \varepsilon \|V\|_2\|f'\|_2+{C}(\varepsilon)\|V\|_2\|f\|_2\:,
\end{equation*}
where $\varepsilon\|V\|_2$ can be made arbitrarily small.
Thus, for any non-negative $V\in L^2(0,1)$, we may conclude that all proper dissipative extensions of the dual pair $S+i V$ and $S-i V$ are given by $S_c+i V$ by virtue of Theorem \ref{prop:weaklyperturbed}.  
\begin{remark} Using that $V$ is $S^*$-bounded with relative bound equal to zero, we have in particular that $V$ is $S_c$-bounded with relative bound equal to zero as well. Thus, by the Hess-Kato Theorem (\cite[Corollary 1]{Hess-Kato})
\begin{equation*}
-(S_c+iV)^*=-(S_c)^*+iV\:.
\end{equation*}
By Proposition \ref{thm:bally}, we have that $-(S_c)^*$ is dissipative, which makes $-(S_c)^*+iV$ dissipative. By the same proposition, we therefore may conclude that $S_c+iV$ is maximally dissipative.
\end{remark}
\end{example}

\subsubsection{\bf{Differential operators with dissipative potentials}} \label{sec:diffdisspot}
For any $n\in\N$, let $p_0^n$ be the symmetric differential operator defined as follows
\begin{align*}
p_0^n:\quad\mathcal{D}(p_0^n)=\mathcal{C}_0^\infty(0,1),\quad f\mapsto i^nf^{(n)}\:,
\end{align*}
where $f^{(n)}$ denotes the $n^{\text{th}}$ derivative of $f$.
  Moreover, let $W\in L^2_{\text{loc}}(0,1)$ be a locally square-integrable potential function with $W\geq 0$ almost everywhere.
  This means that the dual pair of operators
  \begin{align} \label{eq:dualeins}
  A_0:\quad\mathcal{D}(A_0)=\mathcal{C}_0^\infty(0,1),\quad \left(A_0f\right)(x)=i^nf^{(n)}(x)+iW(x)f(x)
  \end{align}
  and
    \begin{align} \label{eq:dualzwei}
  \widetilde{A}_0:\quad\mathcal{D}(\widetilde{A}_0)=\mathcal{C}_0^\infty(0,1),\quad
  \left(\widetilde{A}_0f\right)(x)=i^nf^{(n)}(x)-iW(x)f(x)
  \end{align}
  is well defined. Moreover, their closures $A:=\overline{A_0}$ and $\widetilde{A}:=\overline{\widetilde{A}_0}$ have the common core property by construction. In Theorem \ref{prop:telemann}, the operator $V$ is defined as $\frac{A-\widetilde{A}}{2i}$ on a common core $\mathcal{D}\subset(\mathcal{D}(A)\cap\mathcal{D}(\widetilde{A}))$ and we choose $\mathcal{D}=\mathcal{C}_c^\infty(0,1)$. Since $V$ is already essentially selfadjoint,  this implies that the Kre\u\i n extension of $V$ coincides with its closure  $V_K=\overline{V}$ and is given by the maximal multiplication operator by the function $W(x)$. Thus, $V_K^{1/2}$ is given by
\begin{align*}
V_K^{1/2}:\qquad\mathcal{D}(V_K^{1/2})=\left\{f\in L^2(0,1): \int_0^1 W(x)|f(x)|^2\text{d}x<\infty\right\},\quad
\left(V_Kf\right)(x)=\sqrt{W(x)}f(x)\:.
\end{align*}
Moreover, it can be easily shown that the domains of $A^*$ and $\widetilde{A}^*$ are given by
\begin{align*}
\widetilde{A}^*:&\quad\mathcal{D}(\widetilde{A}^*)=\left\{f\in L^2(0,1); f\in H^n_{\text{loc}}(0,1); i^nf^{(n)}+iWf\in L^2\right\},\quad
f\mapsto i^nf^{(n)}+iWf\:,\\
{A}^*:&\quad\mathcal{D}({A}^*)=\left\{f\in L^2(0,1); f\in H^n_{\text{loc}}(0,1); i^nf^{(n)}-iWf\in L^2\right\},\quad
f\mapsto i^nf^{(n)}-iWf\:,
\end{align*}
with the understanding that $f^{(n)}$ denotes the $n^{\text{th}}$ weak derivative of $f$.
By Theorem \ref{prop:telemann}, the operator $A_\mathcal{V}$ (cf. Definition \ref{def:subspaceextension})
is maximally dissipative, only if $\mathcal{V}\subset\mathcal{D}(V_K^{1/2})$. Thus for any $v\in\mathcal{V}$ this implies that
\begin{equation} \label{eq:convpotential}
\int_0^1|v(x)|^2W(x)\text{d}x<\infty
\end{equation}
and since $v\in\mathcal{D}(\widetilde{A}^*)\subset L^2(0,1)$, which implies that $i^nv^{(n)}+iWv\in L^2(0,1)$, it follows that
\begin{align} \label{eq:convimag}
\overline{v(x)}i^nv^{(n)}(x)+i|v(x)|^2W(x)\in L^1(0,1)
\end{align}
 from which -- together with \eqref{eq:convpotential} and an application of the reverse triangle inequality --  it follows that
\begin{equation*}
\int_0^1\left|\overline{v(x)}i^nv^{(n)}(x)\right|\text{d}x<\infty\:,
\end{equation*}
i.e. $\overline{v}v^{(n)}\in L^1(0,1)$.
Hence, given that $v\in\mathcal{D}(V_K^{1/2})$ the necessary and sufficient condition for $A_\mathcal{V}$ to be dissipative
\begin{equation*}
\Imag \langle v,\widetilde{A}^*v\rangle\geq\|W^{1/2}v\|^2\quad\text{for all}\quad v\in\mathcal{V}
\end{equation*}
simplifies to
\begin{equation} \label{eq:principio}
\Imag\langle v, i^n v^{(n)}\rangle\geq 0\quad\text{for all}\quad v\in\mathcal{V}\:.
\end{equation}

\subsubsection{\bf{First order differential operators with singular potentials}} \label{ex:firstorder}
Let us apply the result of the previous subsection to the simplest case $n=1$. For any $\varepsilon>0$, any $x_0\in (0,1)$ and any $v\in H^1_{\text{loc}}(0,1)$ we have that
\begin{equation*}
|v(\varepsilon)|^2=|v(x_0)|^2-2\Imag\int_\varepsilon^{x_0}\overline{v(x)}iv'(x)\text{d}x
\end{equation*}
and since $\overline{v}v'\in L^1$, we have by an explicit calculation
\begin{equation*}
\lim_{\varepsilon\downarrow 0}|v(\varepsilon)|^2=\lim_{\varepsilon\downarrow 0}\left(|v(x_0)|^2-2\Imag\int_\varepsilon^{x_0}\overline{v(x)}iv'(x)\text{d}x\right)=|v(x_0)|^2-2\Imag\int_0^{x_0}\overline{v(x)}iv'(x)\text{d}x\:.
\end{equation*}
The same reasoning can be applied to show the existence of $\lim_{\varepsilon\downarrow 0}|v(1-\varepsilon)|^2$, which shows that $|v|^2$ is continuous up to the boundary of the interval. Defining, at least formally,
\begin{equation*}
|v(0)|^2:=\lim_{\varepsilon\downarrow 0}|v(\varepsilon)|^2\quad\text{and}\quad|v(1)|^2:=\lim_{\varepsilon\downarrow 0}|v(1-\varepsilon)|^2
\end{equation*}
we get that
\begin{equation} \label{eq:firstorderboundary}
\Imag\langle v,iv'\rangle=\frac{1}{2}\left(|v(1)|^2-|v(0)|^2\right)\quad\text{for all}\quad v\in H^1_{\text{loc}}(0,1): \overline{v}v'\in L^1\:.
\end{equation}
Let us now consider a few different potentials:
\begin{example}
Let $\boxed{1/2\leq \alpha < 1}$ and let the potential function be given by $W(x)=\frac{1-\alpha}{x^\alpha}$, where the numerator $(1-\alpha)$ is chosen for convenience (the case $0<\alpha<1/2$ has been covered in Example \ref{ex:donnerkuchen}).
By an explicit calculation, it can be shown that
\begin{equation*}
\mathcal{D}(\widetilde{A}^*)=\mathcal{D}(A)\dot{+}\spann\left\{\exp(-x^{1-\alpha}),\exp(-x^{1-\alpha})\int_0^x \exp(2t^{1-\alpha})\text{d}t\right\}
\end{equation*}
and it is easy to see that 
\begin{equation*}
\mathcal{D}(\widetilde{A}^*)//\mathcal{D}(A)=\spann\left\{\exp(-x^{1-\alpha}),\exp(-x^{1-\alpha})\int_0^x \exp(2t^{1-\alpha})\text{d}t\right\}\subset\mathcal{D}(V_K^{1/2})=\mathcal{D}(x^{-\frac{\alpha}{2}})\:,
\end{equation*}
where the last inclusion is guaranteed by the choice $\alpha<1$.
A standard linear transformation shows that it is possible to define two vectors $\phi,\psi\in\mathcal{D}(\widetilde{A}^*)//\mathcal{D}(A)$ such that
\begin{equation*}
\mathcal{D}(\widetilde{A}^*)=\mathcal{D}(A)\dot{+}\spann\{\phi,\psi\}
\end{equation*}
and $\phi,\psi$ satisfy the boundary conditions $$\psi(0)=1, \psi(1)=0,  \phi(0)=0, \phi(1)=1\:.$$
Thus, if we choose two complex numbers $(c_1,c_2)\in\C^2\setminus\{(0,0)\}$ in order to parametrize all one-dimensional proper extensions of $(A,\widetilde{A})$ as
\begin{align*}
A_{c_1,c_2}:\quad\mathcal{D}(A_{c_1,c_2})=\mathcal{D}(A)\dot{+}\spann\{c_1\phi+c_2\psi\},\quad
A_{c_1,c_2}=\widetilde{A}^*\upharpoonright_{\mathcal{D}(A_{c_1,c_2})}
\end{align*}
and plug $v_{c_1,c_2}:=c_1\phi+c_2\psi$ into \eqref{eq:firstorderboundary}, we get the condition that
\begin{equation*}
\Imag\langle v_{c_1,c_2},i v_{c_1,c_2}'\rangle=\frac{1}{2}\left(|c_1|^2-|c_2|^2\right)\geq 0\:,
\end{equation*}
i.e. $|c_1|\geq |c_2|$. Thus, we can parametrize all maximally dissipative proper extensions using only one complex parameter $c=c_2/c_1$ with $|c|\leq 1$ and get $\{A_c:|c|\leq 1\}$, where
\begin{align*}
A_{c}:\quad\mathcal{D}(A_{c})=\mathcal{D}(A)\dot{+}\spann\{\phi+c\psi\},\quad
A_{c}=\widetilde{A}^*\upharpoonright_{\mathcal{D}(A_{c})}
\end{align*}
as a complete description of the set of all proper maximally dissipative extensions. 
\end{example}
Let us now consider examples, where the singularity of the potential is of ``same strength" as the differential operator ($\alpha=1$).
\begin{example}
Let $\boxed{0<\gamma<1/2}$ and consider the potential
\begin{equation*} \label{eq:potential}
 W(x)=\frac{\gamma}{1-x}. 
\end{equation*} 
 Note that this is equivalent to considering the operator $-i\frac{\text{d}}{\text{d}y}+i\frac{\gamma}{y}$ after the coordinate change $(1-x)\mapsto y$, which leads to a change of sign in front of the differential part of the operator, changing the situation significantly compared to Example \ref{ex:eins}. \\In this case, a calculation shows that for our range of $\gamma$, we have
\begin{equation*} \mathcal{D}(\widetilde{A}^*)=\mathcal{D}(A)\dot{+}\spann\{(1-x)^\gamma,(1-x)^{1-\gamma}\}\:.
\end{equation*}
Since $0<\gamma<1/2$, it is true that 
\begin{equation*}
\spann\{(1-x)^\gamma,(1-x)^{1-\gamma}\}\subset \mathcal{D}(V_K^{1/2})=\mathcal{D}\left(\frac{1}{\sqrt{1-x}}\right)
\end{equation*}
and $\dim \ker A^*=1$, all proper dissipative extensions of $A$ will be at most one-dimensional extensions, i.e. of the form
\begin{equation*}
\mathcal{D}(A_{c_1,c_2}):=\mathcal{D}(A)\dot{+}\spann\{c_1(1-x)^\gamma+c_2(1-x)^{1-\gamma}\}\:,
\end{equation*}
where $(c_1,c_2)\in\C^2\setminus\{(0,0)\}$. Plugging $v_{c_1,c_2}:=c_1(1-x)^\gamma+c_2(1-x)^{1-\gamma}$  into Equation \eqref{eq:firstorderboundary}, we get the condition
\begin{equation} \label{eq:immernegativ}
\Imag \langle v_{c_1,c_2},iv_{c_1,c_2}'\rangle=-\frac{|c_1+c_2|^2}{2}\geq 0\:,
\end{equation}
which is satisfied if and only if $c_1=-c_2$. Thus, there exists a unique proper maximally dissipative extension of the dual pair $(A,\widetilde{A})$, which is given by
\begin{align*}
A':\quad\mathcal{D}(A')=\mathcal{D}(A)\dot{+}\spann\{(1-x)^\gamma-(1-x)^{1-\gamma}\},\quad A'=\widetilde{A}^*\upharpoonright_{\mathcal{D}(A')}\:.
\end{align*}
This is an example of a dual pair $(A,\widetilde{A})$ with a unique proper maximally dissipative extension, which does not satisfy the assumptions of Corollary \ref{coro:uniquediss}.\\
Next, let us compute the spaces $\mathcal{W}_+, \mathcal{W}_0$ and $\mathcal{W}_-$ as defined in Theorem \ref{prop:quadform}.
Since the form $q$ as defined in Equation \eqref{eq:form} is given by
\begin{equation*}
q(v)=\Imag\langle v,iv'\rangle=\frac{1}{2}(|v(1)|^2-|v(0)|^2)
\end{equation*}
and is non-positive for $v\in\spann\{(1-x)^\gamma,(1-x)^{1-\gamma}\}$ by virtue of Equation \eqref{eq:immernegativ}, we have found  the maximizer of $\langle v,Mv\rangle$ which corresponds to the eigenvalue zero:
\begin{equation*}
\mathcal{W}_0=\ker M=\spann\{(1-x)^\gamma-(1-x)^{1-\gamma}\}
\end{equation*}
and -- using the Gram-Schmidt procedure -- we compute
\begin{equation*}
\mathcal{W}_-=\spann\{\underbrace{(4\gamma^2-8\gamma-5)(1-x)^\gamma-(4\gamma^2-8\gamma+3)(1-x)^{1-\gamma}}_{=:w_-}\}
\end{equation*} 
\begin{equation*}
\text{with eigenvalue} \quad
\lambda_-=\frac{\langle w_-,Mw_-\rangle}{\langle w_-,w_-\rangle}=\frac{\frac{1}{2}{\left(|w_-(1)|^2-|w_-(0)|^2\right)}}{\int_0^1|w_-(x)|^2\text{d}x}=-\frac{2}{-4\gamma^2+4\gamma+7}\:.
\end{equation*} \label{ex:zwei}
\end{example}
\begin{example} \label{ex:eins}
Let $\boxed{0<\gamma<1/2}$ and consider the potential
\begin{equation*} \label{eq:criticpotential}
W(x)=\frac{\gamma}{x}\:.
\end{equation*}
In this case, a calculation shows that $\mathcal{D}(\widetilde{A}^*)=\mathcal{D}(A)\dot{+}\spann\{x^{-\gamma},x^{1+\gamma}\}$.
This is an example, for which Corollary \ref{coro:uniquediss} applies, since $\ker{\widetilde{A}^*}=\spann\{x^{-\gamma}\}$ has trivial intersection with $\mathcal{D}(V^{1/2}_K)=\{f\in L^2(0,1), \int_0^1 |f(x)|^2 x^{-1}\text{d}x<\infty\}$. Hence, the only possible candidate for a proper maximally dissipative extension for the dual pair $(A,\widetilde{A})$ is the operator $\widehat{A}$, which is given by
\begin{align*}
\widehat{A}:\quad\mathcal{D}(\widehat{A})=\mathcal{D}(A)\dot{+}\spann\{x^{1+\gamma}\},\quad
\widehat{A}=\widetilde{A}^*\upharpoonright_{\mathcal{D}(\widehat{A})}\:.
\end{align*}
By Proposition \ref{thm:friendly}, it is already clear that $\widehat{A}$ has to be a proper maximally dissipative extension. This can also be verified explicitely by by plugging $v(x):=x^{1+\gamma}$ into Condition \eqref{eq:firstorderboundary}. \\
 In this concrete case, we have that $\mathcal{W}_0=\mathcal{W}_-=\{0\}$ and $\mathcal{W_+}=\spann\{x^{1+\gamma}\}$. A short calculation shows that the corresponding eigenvalue of $M$ is given by $$\lambda_+=\frac{\langle x^{1+\gamma},Mx^{1+\gamma}\rangle}{\langle x^{1+\gamma},x^{1+\gamma}\rangle}=\frac{3}{2}+\lambda\:.$$

\end{example}
\subsubsection{\bf{A second order example}} \label{sec:secondorder}
Let us now apply our results to an example, where the operator $V$ as defined in the statement of Theorem \ref{prop:telemann} is not essentially selfadjoint.
To this end, consider the dual pair of operators given by 
\begin{align*}
A_0:\qquad\mathcal{D}(A_0)=\mathcal{C}_c^\infty(0,1),\quad
\left(A_0f\right)(x)=-if''(x)-\gamma\frac{f(x)}{x^2},\\
\widetilde{A}_0:\qquad\mathcal{D}(\widetilde{A}_0)=\mathcal{C}_c^\infty(0,1),\quad
\left(\widetilde{A}_0f\right)(x)=if''(x)-\gamma\frac{f(x)}{x^2}\:.
\end{align*}
Since we have
\begin{align*}
\Imag\langle f, A_0f\rangle=\Imag \int_0^1 \overline{f(x)} \left(-if''(x)-\gamma \frac{f(x)}{x^2}\right)\text{d}x=\int_0^1|f'(x)|^2\text{d}x
\end{align*}
for all $f\in\mathcal{C}_c^\infty(0,1)$, we can estimate $\Imag\langle f,A_0f\rangle$ from below by the lowest eigenvalue of the Dirichlet-Laplacian on the unit interval, which is $\pi^2$, i.e.\
\begin{equation} \label{eq:natale}
\Imag\langle f,A_0f\rangle\geq\pi^2\|f\|^2\quad\text{for all}\quad\psi\in\mathcal{D}(A_0)\:.
\end{equation}
Now, define $A:=\overline{A_0}$ and $\widetilde{A}:=\overline{\widetilde{A}_0}$, which means that the dual pair $(A,\widetilde{A})$ has the common core property by construction. Also,  \eqref{eq:natale} implies in particular that $0\in\widehat{\rho}(A)$. By a simple calculation, it can be shown that the operator $\widetilde{A}^*$ is given by:
\begin{align*}
\mathcal{D}(\widetilde{A}^*)=\left\{f\in H^2_{loc}(0,1): \int_0^1\left|-if''(x)-\gamma\frac{f(x)}{x^2}\right|^2\text{d}x<\infty\right\},\:
\left(\widetilde{A}^*f\right)(x)=-if''(x)-\gamma\frac{f(x)}{x^2}\:.
\end{align*}
A calculation, using Formula \eqref{eq:malamud} for $\lambda=0$, yields 
\begin{equation}
\mathcal{D}(\widetilde{A}^*)=\mathcal{D}(A)\dot{+}\spann\left\{x^\omega, x^{\overline{\omega}+2}\right\}\:,
\end{equation}
where $\omega:=(1+\sqrt{1+4i\gamma})/{2}$. Here we have assumed that $\boxed{\gamma\geq \sqrt{3}}$. This choice for $\gamma$ ensures that $\dim \ker \widetilde{A}^*=\dim \ker A^*=1$, which will make our calculations simpler. 
Also, observe that $\widetilde{A}^*=JA^*J$, where the conjugation operator $J$ is defined as $(Jf)(x):=\overline{f(x)}$. From this it immediately follows that $\mathcal{D}(A^*)=J\mathcal{D}(\widetilde{A}^*)=\{f:\overline{f}\in\mathcal{D}(\widetilde{A}^*)\}$. 
Now, let us apply the result of Theorem \ref{prop:telemann} in order to construct maximally dissipative extensions of the dual pair $(A,\widetilde{A})$. Let $\mathcal{D}=\mathcal{C}_c^\infty(0,1)$, which is a common core for $A$ and $\widetilde{A}$ and define $V:=\frac{1}{2i}(A-\widetilde{A})\upharpoonright_\mathcal{D}$, which is given by
\begin{align*}
V:\qquad\mathcal{D}(V)=\mathcal{C}_c^\infty(0,1),\quad
f\mapsto-f''\:.
\end{align*}
As the norm induced by $\|\cdot\|_V:=\|\cdot\|+\langle \cdot,V\cdot\rangle$ is the $H^1$-norm, closing $\mathcal{D}(V)=\mathcal{C}_c^\infty(0,1)$ with respect to $\|\cdot\|_V$ yields that $\mathcal{D}(V_F^{1/2})=H^1_0(0,1)$. Moreover, since $\ker V^*=\spann\{1,x\}$ and since by \eqref{eq:kreinformdomain}, we have $\mathcal{D}(V_K^{1/2})=\mathcal{D}(V_F^{1/2})\dot{+}\ker V^*$ it is clear that $\mathcal{D}(V_K^{1/2})=H^1(0,1)$ and moreover that 
\begin{equation} \label{eq:kreindiff}
\|V_K^{1/2}f\|^2=\|V_F^{1/2}\left[f(x)-f(0)-x(f(1)-f(0))\right]\|^2=\|f'\|^2-|f(1)-f(0)|^2\:,
\end{equation}
where the first equality follows from the decomposition \eqref{eq:kreinformdomain} and the second from an explicit calculation.
Using this, we can show that the form $q(v):=\Imag\langle v,\widetilde{A}^*v\rangle-\|V_K^{1/2}v\|^2$ defined on $\mathcal{D}(\widetilde{A}^*)//\mathcal{D}(A)=\spann\{x^\omega,x^{\overline{\omega}+2}\}$ is given by
\begin{equation*}
q(v)=-\Real\left(\overline{v(1)}v'(1)\right)+|v(1)|^2\:.
\end{equation*}
By Lemma \ref{prop:dampfnudel}, any maximally dissipative proper extension of $(A,\widetilde{A})$ can be parametrized by a one-dimensional subspace of $\spann\{x^\omega,x^{\overline{\omega}+2}\}$. A convenient basis for this is given by the two functions 
 \begin{align} \label{eq:functions}
 \psi(x):=\frac{(2+\overline{\omega_+})x^{\omega_+}-\omega_+x^{\overline{\omega_+}+2}}{2+\overline{\omega_+}-\omega_+}\quad\text{and}\quad\phi(x):=\frac{-x^{\omega_+}+x^{\overline{\omega_+}+2}}{2+\overline{\omega_+}-\omega_+}\:,
 \end{align}
 which satisfy the boundary conditions $\psi(1)=1, \psi'(1)=0, \phi(1)=0$ and $\phi'(1)=1$.\\
 Now define $\xi_\rho:=\rho\psi+\phi$, where $\rho\in \C$ has to be determined such that $q(\xi_\rho)\geq 0$. A short explicit calculation shows that this is the case if and only if
 \begin{equation*}
 \left|\rho-\frac{1}{2}\right|\geq \frac{1}{2}\:,
 \end{equation*}
 i.e. if and only if $\rho$ lies in the exterior of the open circle with radius and center point $\frac{1}{2}$. Since $q(\psi)=1>0$, we have that $\xi_\infty:=\psi$ describes a maximally dissipative extension as well.
 Thus the set of all proper maximally dissipative extensions of $(A,\widetilde{A})$ is given by
 \begin{align} \label{eq:arho}
 A_\rho:\qquad\mathcal{D}(A_\rho)=\mathcal{D}(A)\dot{+}\spann\{\xi_\rho\},\quad
 A_\rho=\widetilde{A}^*\upharpoonright_{\mathcal{D}(A_\rho)}\:,
 \end{align}
 where
 \begin{equation} \label{eq:setcondition}
 \rho\in \left\{z\in\C, \left|z-\frac{1}{2}\right|\geq \frac{1}{2}\right\}\cup\{\infty\}\:.
 \end{equation}

\section{Stability of the numerical range}
Let us now prove a useful result that allows us to estimate the lower bound of the imaginary part of the numerical range of the extensions of a dual pair $(A,\widetilde{A})$:
\begin{lemma} Let the dual pair $(A,\widetilde{A})$ satisfy the assumptions of Theorem \ref{prop:telemann} and let $\mathcal{V}$ be a subspace of $\mathcal{D}(\widetilde{A}^*)//\mathcal{D}(A)$ such that $\mathcal{D}(A_\mathcal{V})$ is a proper dissipative extension of the dual pair $(A,\widetilde{A})$. Moreover, for $v\in\mathcal{V}$, let $q(v):=\Imag\langle v,\widetilde{A}^*v\rangle-\|V_K^{1/2}v\|^2$. Then, it is true that
\begin{equation*}
\Imag\langle (f+v),A_\mathcal{V}(f+v)\rangle=\|V_K^{1/2}(f+v)\|^2+q(v)\geq\|V_K^{1/2}(f+v)\|^2\quad\text{for all}\quad f\in\mathcal{D}(A), v\in\mathcal{V}\:.
\end{equation*} \label{prop:cernohorsky}
\end{lemma}
\begin{proof}
Let $f\in\mathcal{D}$ and $v\in\mathcal{V}$. As in the proof of Theorem \ref{prop:telemann}, we use Lemma \ref{lemma:stollen}, from which we know that it is sufficient to check the assertion for such $f$ and $v$. From an explicit calculation, we get
\begin{align} \label{eq:monteverdi}
\Imag\langle(f&+v),A_\mathcal{V}(f+v)\rangle=\Imag\langle(f+v),\widetilde{A}^*(f+v)\rangle\notag=\Imag\langle f,Af\rangle+\Imag\langle v,\widetilde{A}^*v\rangle\\&+\Imag(\langle f,\widetilde{A}^*v\rangle+\langle v,\widetilde{A}^*f\rangle)
=\Imag\langle f,A f\rangle+q(v)+\|V_K^{1/2}v\|^2+\Imag(\langle f,\widetilde{A}^*v\rangle+\langle v,\widetilde{A}^*f\rangle)\:.
\end{align}
Now, we can use that $\Imag\langle f,Af\rangle=\langle f,V f\rangle$, which implies in particular that $f\in\mathcal{D}\subset\mathcal{D}(V_K)\subset \mathcal{D}(V_K^{1/2})$ since $V_K$ is a selfadjoint extension of $V$. Thus, we have that
\begin{equation*}
\Imag\langle f,Af\rangle=\langle f,Vf\rangle=\|V_K^{1/2}f\|^2
\end{equation*}
and another calculation -- similar as in the proof of Theorem \ref{prop:telemann} -- shows that
\begin{equation*}
\Imag (\langle f,\widetilde{A}^*v\rangle+\langle v,\widetilde{A}^*f\rangle)=2\Real\langle V_K^{1/2}f,V_K^{1/2}v\rangle\:.
\end{equation*}
Plugging these two identities back into \eqref{eq:monteverdi} yields
\begin{align*}
\Imag\langle (f+v),\widetilde{A}^*(f+v)\rangle&=\|V_K^{1/2}f\|^2+2\Real\langle V_K^{1/2}f,V_K^{1/2}v\rangle+\|V_K^{1/2}v\|^2+q(v)\\&=\|V_K^{1/2}(f+v)\|^2+q(v)\:.
\end{align*}
Since by Theorem \ref{prop:telemann} we have that $q(v)\geq 0$ for all $v\in\mathcal{V}$ it trivially follows that
\begin{equation*}
\Imag\langle f+v,A_\mathcal{V}(f+v)\rangle\geq \|V_K^{1/2}(f+v)\|^2
\end{equation*}
for all $f\in\mathcal{D}(A)$ and $v\in\mathcal{V}$.
\end{proof}
\begin{example} As a first example, consider the dual pair $(A,\widetilde{A})$ from Subsection \ref{sec:secondorder}, with the maximally dissipative extensions $A_\rho$ as described in \eqref{eq:arho} and \eqref{eq:setcondition}. Again, it suffices to find a lower bound of $\Imag\langle f+v,\widetilde{A}^*(f+v)\rangle$ for all $f\in\mathcal{C}_c^\infty(0,1)$ and all $v\in\spann\{\xi_\rho\}$, where $\xi_\rho$ was defined in Subsection \ref{sec:secondorder}. Observe that
\begin{equation} \label{eq:imaginaryform}
\Imag \langle f+v,{A}_\rho(f+v)\rangle=\|f'+v'\|^2-\frac{\Real(\rho)}{|\rho|^2}|v(1)|^2=:\mathfrak{a}(f+v)
\end{equation}
and $\mathcal{C}_c^\infty(0,1)\dot{+}\spann\{\xi_\rho\}\subset \mathfrak{C}$, where $\mathfrak{C}:=\{f\in H^1(0,1): f(0)=0\}$. For the special cases $\rho=0$ and $\rho=\infty$, we have $$\Imag \langle f+v,A_\rho(f+v)\rangle=\|f'+v'\|^2=:\mathfrak{a}(f+v)\:.$$
 Now, since $\mathfrak{C}$ equipped with the norm induced by $\mathfrak{a}$ is a Hilbert space, this implies that $\Imag\langle f+v,A_\mathcal{V} (f+v)\rangle\geq \lambda_\rho \|f+v\|^2$, where
$\lambda_\rho$ is the lowest eigenvalue of the selfadjoint operator $S_\rho$ associated to $(\mathfrak{a},\mathfrak{C})$. This operator is given by
\begin{align*}
S_\rho:\qquad\mathcal{D}(S_\rho)=\left\{f\in H^2(0,1): f(0)=0\:\:\text{and}\:\:f'(1)=\frac{\Real(\rho)}{|\rho|^2}f(1)\right\},\quad
f\mapsto-f''\:,
\end{align*}
\vspace{-2mm}
with the understanding that the case $\rho=0$ corresponds to a Dirichlet boundary condition at one. As it is not difficult to solve the eigenvalue equation $S_\rho f=\lambda_\rho f$, where $\lambda_\rho$ is the smallest eigenvalue of $S_\rho$, one finds that $\lambda_\rho$ is given by $\lambda_\rho=z^2$, where $z$ is the smallest positive solution of the transcendental equation 
\begin{equation*}
\frac{\tan z}{z}=\frac{|\rho|^2}{\Real(\rho)}\:,
\end{equation*}
where $\rho\in\{z\in\C: z\neq 0, \Real(z)=0\}$ corresponds to the singularity of $\frac{\tan z}{z}$ at $z=\frac{\pi}{2}$.\\
For $\Real(\rho)<0$, this means in particular that $\Imag\langle f+v,A_\rho(f+v)\rangle\geq \frac{\pi^2}{4}\|f+v\|^2$ as can easily be seen from the fact that $(\tan z)/ z$ is positive in $[0,\pi/2)$ and non-positive in $(\pi/2,\pi]$.
\begin{remark} In this example, the estimate on the lower bound of the imaginary parts is also sharp. This follows from the fact that closing $\mathcal{C}_0^\infty(0,1)\dot{+}\spann\{\xi_\rho\}$ with respect to the norm induced by $\mathfrak{a}$ yields $\mathfrak{C}$ for $\rho\neq 0$ and closing $\mathcal{C}_0^\infty(0,1)\dot{+}\spann\{\xi_0\}$ with respect to the $H^1$-norm yields $H^1_0(0,1)$. 
\end{remark}

\end{example}
\begin{theorem} \label{coro:stablerange}
Let the dual pair $(A,\widetilde{A})$ satisfy the same conditions as in Theorem \ref{prop:telemann}. If in addition we have that $\mathcal{V}\subset\mathcal{D}(V_F^{1/2})$, we get that the imaginary part of the numerical range stays stable, i.e.
\begin{equation*}
\inf_{z\in\mathcal{N}_A}\text{\emph{Im}} z=\inf_{z\in\mathcal{N}_{A_\mathcal{V}}}\text{\emph{Im}} z\:,
\end{equation*}
where $\mathcal{N}_C$ denotes the numerical range of an operator $C$ and $A_\mathcal{V}$ is the extension of $A$ as described in Definition \ref{def:subspaceextension}.
This is true in particular for any dissipative extension of a dual pair operator $(A,\widetilde{A})$, where the associated operator $V$ is essentially selfadjoint.
\end{theorem}
\begin{proof}
For $f\in\mathcal{D}(A)\cap\mathcal{D}(\widetilde{A})$, we have that $f\in\mathcal{D}(V)\subset\mathcal{D}(V_F^{1/2})$. Now, since by assumption $\mathcal{V}\subset\mathcal{D}(V_F^{1/2})$, we get by virtue of Lemma \ref{prop:cernohorsky} that
\begin{equation} \label{eq:jroth}
\Imag\langle (f+v),\widetilde{A}^*(f+v)\rangle\geq\|V_K^{1/2}(f+v)\|^2=\|V_F^{1/2}(f+v)\|^2\:,
\end{equation}
for all $f\in\mathcal{D}(A)\cap\mathcal{D}(\widetilde{A})$ and for all $v\in\mathcal{V}$. Using that for all $f\in\mathcal{D}(A)\cap\mathcal{D}(\widetilde{A})$ we have that 
\begin{equation*}
\Imag\langle f,A f\rangle=\langle f,V f\rangle\:,
\end{equation*}
which implies that 
\begin{equation*}
\inf_{z\in\mathcal{N}_A}\Imag z=\inf_{x\in\mathcal{N}_{V}}x=\inf_{x\in\mathcal{N}_{V_F}}x\:,
\end{equation*}
where the last equality follows from the fact that the numerical range of the Friedrichs extension of a semibounded operator stays stable. Using Inequality \eqref{eq:jroth}, we therefore get
\begin{equation*}
\inf_{z\in\mathcal{N}_{A_\mathcal{V}}}\text{{Im}} z\geq \inf_{x\in\mathcal{N}_{V_F}}x=\inf_{z\in\mathcal{N}_{A}}\text{{Im}} z\:,
\end{equation*}
which together with the trivial estimate for taking the infimum over a larger set
$$ \inf_{z\in\mathcal{N}_{A_\mathcal{V}}}\text{{Im}} z\leq\inf_{z\in\mathcal{N}_{A}}\text{{Im}} z
$$
yields the theorem.
\end{proof}

\begin{example} As an example, consider the operators $(A_0,\widetilde{A}_0)$ as described in Section \ref{sec:diffdisspot}, \eqref{eq:dualeins} and \eqref{eq:dualzwei}.
Since the operator $V=\frac{1}{2i}(A_0-\widetilde{A}_0)$ is given by
\begin{align*}
V:\qquad\mathcal{D}(V)=\mathcal{C}_c^\infty(0,1),\quad
(Vf)(x)=W(x)f(x)\:,
\end{align*}
\vspace{-2mm}
which is essentially selfadjoint, we get that $V_K^{1/2}=V_F^{1/2}=\overline{V}^{1/2}$ is the maximal multiplication operator by $\sqrt{W(x)}$. 
Hence by virtue of Theorem \ref{coro:stablerange}, we get that for any proper maximally dissipative extension $A_\mathcal{V}$, we have
\begin{equation*}
\Imag\langle f+v, A_\mathcal{V} (f+v)\rangle\geq w\|f+v\|^2\:,
\end{equation*}
where $w:=\text{essinf}_{x\in(0,1)} W(x)=\inf_{f\in\mathcal{D}(A):\|f\|=1}\langle f,Af\rangle$.
\end{example}

\end{document}